\documentclass[a4paper,10pt]{amsart}
\usepackage{a4wide}
\usepackage{url}
\usepackage{amssymb}
\usepackage{color}
\usepackage{booktabs}   
\usepackage{float}      
\usepackage{lscape}     
\usepackage{changepage} 
\usepackage{hyperref}
\usepackage[english,linesnumbered]{algorithm2e}
\usepackage{subcaption}
\usepackage[ddmmyyyy,hhmmss]{datetime}
\usepackage[utf8]{inputenc}
\usepackage{srcltx}

\numberwithin{equation}{section}

\newcommand{\field}[1]{\mathbb{{#1}}}
\newcommand{\bs}[1]{{\boldsymbol{#1}}}
\newcommand{\R}{\field{R}}
\newcommand{\C}{\field{C}}
\newcommand{\N}{\field{{N}}}
\newcommand{\Z}{\field{{Z}}}
\newcommand{\Q}{\field{{Q}}}
\newcommand{\K}{\field{{K}}}
\newcommand{\intpart}[1]{\left\lfloor#1\right\rfloor}

\newcommand{\disc}{\Delta}
\newcommand{\dq}{\disc_q}
\renewcommand{\d}{\,\mathrm{d}}
\newcommand{\Imm}{\mathrm{Im}}
\newcommand{\Ree}{\mathrm{Re}}
\newcommand{\Rp}{U}
\newcommand{\tc}{\bs{1}}
\newcommand{\eul}{\varphi}

\newcommand{\Pmod}[1]{\mathop{[#1]}}
\newcommand{\mltc}{\multicolumn}
\DeclareMathOperator{\odd}{odd}
\DeclareMathOperator{\even}{even}

\newtheorem{theorem}{Theorem}[section]
\newtheorem*{theorem*}{Theorem}
\newtheorem{lemma}[theorem]{Lemma}
\newtheorem*{lemma*}{Lemma}

\newtheorem{corollary}[theorem]{Corollary}
\newtheorem*{corollary*}{Corollary}

\theoremstyle{remark}

\newtheorem*{remark*}{Remark}
\newtheorem*{acknowledgements}{Acknowledgements}

\allowdisplaybreaks[4]

\begin{document}
\title{Explicit Short Intervals for Primes in Arithmetic Progressions on GRH}

\author[A.~Dudek]{Adrian W. Dudek}
\address[A.~Dudek]{Mathematical Sciences Institute\\
         The Australian National University}
\email{adrian.dudek@anu.edu.au}

\author[L.~Greni\'{e}]{Lo\"{\i}c Greni\'{e}}
\address[L.~Greni\'{e}]{Dipartimento di Ingegneria gestionale, dell'informazione e della produzione\\
         Universit\`{a} di Bergamo\\
         viale Marconi 5\\
         24044 Dalmine (BG)
         Italy}
\email{loic.grenie@gmail.com}

\author[G.~Molteni]{Giuseppe Molteni}
\address[G.~Molteni]{Dipartimento di Matematica\\
         Universit\`{a} di Milano\\
         via Saldini 50\\
         20133 Milano\\
         Italy}
\email{giuseppe.molteni1@unimi.it}



\begin{abstract}
We prove explicit versions of Cram\'{e}r's theorem for primes in arithmetic progressions, on the assumption
of the generalised Riemann hypothesis.
\end{abstract}

\maketitle

\section{Motivations and results}
\label{sec:A1}
The purpose of this article is to combine techniques from analytic number theory with computation to
furnish explicit short interval results for primes in arithmetic progressions. This is done on the
assumption of the generalised Riemann hypothesis (GRH), and builds on the earlier work of the authors
\cite{DudekGrenieMolteni1}, where the problem was considered without reference to residue classes.

Throughout this paper, unless it is mentioned, we will be assuming GRH to be true. Let $q\in\N$ and
$a\in\Z$ with $(a,q)=1$. Unconditionally, both McCurley~\cite{McCurley} and later Kadiri~\cite{Kadiri3}
proved that, for every positive $\epsilon$ and $q_0$, there exists $\alpha=\alpha(\epsilon,q_0)$ such
that if $\log x\geq \alpha_\epsilon \log^2q$ and $q\geq q_0$, then $[x,e^{\epsilon}x]$ contains a prime
$p$ congruent to $a$ modulo $q$. They provide pairs of explicit values for $\alpha$ and $q_0$ dependent
on the choice of $\epsilon$; Kadiri's work improves on that of McCurley by providing smaller values of
$\alpha$.

Clearly, on the assumption of GRH, the result should improve significantly; Dusart proved in his Ph.D.
Thesis~\cite[Th.~3.7, p.~114]{Dusart} that when $x\geq \max(\exp(\frac{5}{4}q),10^{10})$ one has
\[
\Big|\psi(x;q,a) - \frac{x}{\eul(q)}\Big| \leq \frac{1}{4\pi}\sqrt{x}\log^2 x.
\]
This implies that there is a prime in $[x-h,x+h]$ which is congruent to $a$ modulo
$q$ provided that $h>(\frac{1}{4\pi}+\epsilon)\eul(q)\sqrt{x}\log^2 x$ and $x\geq x_0(q,\epsilon)$ for every
$\epsilon>0$.\\
Recently, in joint work with the second and third authors, Perelli~\cite[Th.~1]{GrenieMolteniPerelli}
proved that there exist absolute (i.e., independent of $x$ and $q$) positive constants $x_0$, $c_1$ and
$c_2$ such that for $x \geq  x_0$ and $c_1\eul(q) \sqrt{x}\log x \leq h \leq x$ one has
\begin{equation}\label{eq:A0}
\pi(x + h;q,a) - \pi(x;q,a) \geq c_2 \frac{h}{\eul(q)\log x}.
\end{equation}
This is in some sense the best result we can hope to prove, but the constants are not explicit.\\
In the present paper we prove the following result.
\begin{theorem}\label{th:A1}
Assume GRH, $q\geq 3$ and $(a,q)=1$. Let $\alpha$, $\delta$, $\rho$, $m$ and $m'$ be as in
Table~\ref{tab:A1} and assume
\[
h \geq \eul(q)(\alpha\log x + \delta \log q + \rho)\sqrt{x}
\]
and $x \geq (m\eul(q)\log q)^2$. Then there is a prime $p$ which is congruent to $a$ modulo $q$ with
$|p-x|<h$. Furthermore, if we assume
\[
h \geq \eul(q)((\alpha+1)\log x + \delta \log q + \rho)\sqrt{x}
\]
and $x \geq (m'\eul(q)\log q)^2$, then there are at least $\sqrt{x}$ such primes.
\end{theorem}
\begin{table}[H]
\caption{Parameters for Theorem~\ref{th:A1}}\label{tab:A1}
\centering
\begin{tabular}{|rrr|rr||rrr|rr|}
  \toprule
 \mltc{1}{|c}{$\alpha$}&\mltc{1}{c}{$\delta$}&\mltc{1}{c }{$\rho$}&\mltc{1}{ c}{$m$}&\mltc{1}{ c||}{$m'$}
&\mltc{1}{|c}{$\alpha$}&\mltc{1}{c}{$\delta$}&\mltc{1}{c }{$\rho$}&\mltc{1}{ c}{$m$}&\mltc{1}{ c|}{$m'$}\\
  \midrule
 1/2     & 1   & 12 &    23 &   46    &     1.253/2 & 0.1 &  7 &   500 & 1500\\
 1/2     & 1/2 &  9 &    86 &  188    &     1       & 0   &  8 &    23 &   46\\
 1/2     & 1/3 &  9 &  1500 & 3500    &     0.9     & 0   &  7 &    31 &   66\\
 1.253/2 & 1   & 14 &    18 &   34    &     0.8     & 0   &  6 &    52 &  120\\
 1.253/2 & 1/2 &  9 &    34 &   74    &     0.7     & 0   &  5 &   200 &  500\\
 1.253/2 & 0.2 &  7 &   110 &  260    &             &     &    &       &     \\
  \bottomrule
\end{tabular}
%
%
\end{table}

The claim of this theorem has the same qualitative behavior in its dependencies on $q$ as what is
predicted in~\eqref{eq:A0}, but the constants $m$ and $m'$ ruling the minimum $x$ are quite large. This
is the effect of the fact that under the hypotheses for the theorem the quotient $h/x$ is for sure
small in the long run for $x$, but this happens uniformly in $q$ only for very large values of $q$. In
fact, when $x$ has its lowest value we have
\[
\frac{h}{x_{\min}}
= \frac{\eul(q)(\alpha\log x_{\min} +\delta \log q + \rho)}{\sqrt{x_{\min}}}
= \frac{2\alpha+\delta + o(1)}{m+o(1)}.
\]
This considerably affects the computations, because they are more effective for smaller $h/x$, and
so we are forced to choose larger values of $m$ and $m'$. This also means that for small $q$ and some
limited range of $x$, extensive numerical tests have to be performed to complete the proof.
\smallskip

From the same general formulas we also deduce the following result.
\begin{theorem}\label{th:A2}
Assume GRH, $q\geq 3$, $(a,q)=1$,
\[
h \geq \eul(q)\Big(\frac{1}{2}\log(q^2x) + 15\Big)\sqrt{x}
\]
and $x \geq (8\eul(q)\log q\log\log q)^2$. Then there is a prime $p$ which is congruent to $a$ modulo $q$
with $|p-x|<h$. Furthermore, if we assume
\[
h \geq \eul(q)\Big(\frac{1}{2}\log(q^2x^3) + 15\Big)\sqrt{x}
\]
and $x \geq (15\eul(q)\log q\log\log q)^2$, then there are at least $\sqrt{x}$ such primes.
\end{theorem}
Theorem~\ref{th:A2} is worse in its dependency of the minimum $x$ on $q$, but the constants are better. As
a consequence, its claims improve on the case $\alpha=1/2$, $\delta=1$ in Theorem~\ref{th:A1} for all
$q\leq \exp(\exp(23/8))\simeq 5\cdot 10^7$ (resp. $q\leq \exp(\exp(46/15))\simeq 2\cdot 10^9$).

Both theorems could be adapted to include the cases $q=1$ and $q=2$, but for them we have already proved
a better result in~\cite{DudekGrenieMolteni1} where the conclusions are proved with $h = \tfrac{1}{2}\log
x + 2$ for any $x\geq 2$.

The conclusions improve significantly if, following Dusart, we select a lower bound for $x$ of
exponential type in terms of $q$. In fact, the same formulas producing Theorems~\ref{th:A1}
and~\ref{th:A2} allow us to prove the following result.
\begin{theorem}\label{th:A3}
Assume GRH, $(a,q)=1$ and $x\geq \exp(q)$. Let
\[
h \geq \frac{\eul(q)}{2}\log(q^2x)\sqrt{x}.
\]
Then for each $q\geq 35$ there is a prime $p$ which is congruent to $a$ modulo $q$ with
$|p-x|<h$. Furthermore, assuming
\[
h \geq \frac{\eul(q)}{2}\log(q^2x^3)\sqrt{x}
\]
and $q\geq 67$, there are at least $\sqrt{x}$ such primes.
\end{theorem}
This claim is always stronger than what we deduce from Dusart's result, apart from the larger minimum value
for $q$.
\medskip

Note that Theorem~\ref{th:A1} (case $\alpha=1/2$, $\delta=1$) shows that the least prime congruent to
$a$ modulo $q$ is lower than
\[
(24^2+o(1))(\eul(q)\log q)^2
\]
%
%
where $o(1)$ is explicit. According to computations in Section~\ref{sec:A5} (see Table~\ref{tab:A3}), the
constant reduces to $21^2+2\cdot 21$ for extremely large values of $q$ but this is notably weaker than the
bound $(\eul(q)\log q)^2$ which has been proved by Lamzouri, Li and
Soundararajan~\cite[Cor.~1.2]{LamzouriLiSoundararajan} for all $q\geq 4$.
\medskip

Also, from Theorem~\ref{th:A1}, one deduces the following explicit version of a quasi-Dirichlet's conjecture
for primes close to squares of integers.
\begin{corollary}\label{cor:A1}
Assume GRH and let $q\geq 1$ and $n\geq 8\eul(q)\log q$. Then the interval
\[
\big(n^2, (n+\eul(q)(12+2\log(qn)))^2\big)
\]
contains a prime which is congruent to $a$ modulo $q$, for every $a$ coprime to $q$.
\end{corollary}
Similar corollaries may be deduced from Theorems~\ref{th:A2} and~\ref{th:A3}.

\section{Functional equation and integral representation}
\label{sec:A2}
Let $\chi$ be a character modulo $q$; let $\chi^*$ be the primitive character inducing $\chi$ and let
$q_\chi$ be its conductor. Let $a_\chi:=(1-\chi(-1))/2$ denote the parity of $\chi$, so that
\[
L(s,\chi) = L(s,\chi^*)\prod_{p|q}\big(1-\chi^*(p)p^{-s}\big)
\]
and
\[
\xi(1-s,\chi^*) = \frac{\tau(\chi^*)}{i^{a_\chi}\sqrt{q_\chi}} \xi(s,\overline{\chi^*})
\]
where we have that
\[
\xi(s,\chi^*) := s(s-1)\Big(\frac{q_\chi}{\pi}\Big)^{\frac{s+a_\chi}{2}}\Gamma\Big(\frac{s+a_\chi}{2}\Big)L(s,\chi^*).
\]
We also let
\begin{equation}\label{eq:A1}
\psi_\chi^{(1)}(x):= \int_0^x \psi_\chi(u)\d u
                   = \sum_{n\leq x} \chi(n)\Lambda(n)(x-n)
\end{equation}
and recall the integral representation
\begin{equation}\label{eq:A2}
\psi_\chi^{(1)}(x)
=  -\frac{1}{2\pi i}\int_{2-i\infty}^{2+i\infty} \frac{L'}{L}(s,\chi) \frac{x^{s+1}}{s(s+1)}\d s
\end{equation}
which holds for all $x \geq 1$. The next lemma gives an alternative
formula for $\psi_{\chi^*}^{(1)}(x)$ based on the representation~\eqref{eq:A2} applied to the
character $\chi^*$.
\begin{lemma}\label{lem:A1}
We have that
\begin{equation}\label{eq:A3}
\psi_{\chi^*}^{(1)}(x)
=  \frac{x^2}{2}\delta_{\chi^*=\tc}
 - \sum_{\rho\in Z_{\chi^*}} \frac{x^{\rho+1}}{\rho(\rho+1)}
 - x r_{\chi^*}
 + r'_{\chi^*}
 + R_{\chi^*}^{(1)}(x)
\end{equation}
where $Z_{\chi^*}$ is the set of nontrivial zeros of $L(s,{\chi^*})$, and $r_{\chi^*}$, $r'_{\chi^*}$ are
the constants
\begin{align*}
r_{\chi^*}
&= \frac{L'}{L}(0,\chi^*)(a_\chi\delta_{\chi^*\neq\tc} + \delta_{\chi^*=\tc}) - (1 + \beta)(1-a_\chi)\delta_{\chi^*\neq \tc},\\
r'_{\chi^*}
&= (1 - \alpha)a_\chi + \frac{L'}{L}(-1,\chi^*)(1-a_\chi),
\intertext{with $\alpha$ and $\beta \in\C$ defined in the proof,}
R_{\chi^*}^{(1)}(x)
&= - \sum_{n=1}^\infty \frac{x^{1-2n-a_\chi}}{(2n+a_\chi)(2n+a_\chi-1)}
  + a_\chi\log x
  - (1-a_\chi)\delta_{\chi^*\neq \tc}\,x\log x,
\end{align*}
$\delta_{\chi^*=\tc}$ is $1$ when $\chi^*=\tc$ and $0$ otherwise, and $\delta_{\chi^*\neq\tc}=1-\delta_{\chi^*=\tc}$.
\end{lemma}
\begin{proof}
The poles of $-\frac{L'}{L}(s,\chi)\frac{x^{s+1}}{s(s+1)}$ at the trivial zeros $s=-2n-a_\chi$ are simple
for every integer $n\geq 1$. Moreover, when $a_\chi= 1$, the pole at $s=0$ is simple and its
contribution to $R_{\chi^*}^{(1)}(x)$ is $-\frac{L'}{L}(0,\chi^*)x$, while the one in $s=-a_\chi=-1$ (i.e.
$n=0$) is a double pole with contribution
\[
1 - \alpha + \log x,
\]
where $-\frac{L'}{L}(-1+\epsilon,\chi^*) =: -\frac{1}{\epsilon} + \alpha + O(\epsilon)$.\\
Lastly, when $a_\chi= 0$, the pole at $s=-1$ is simple and its contribution to $R_{\chi^*}^{(1)}(x)$ is
$\frac{L'}{L}(-1,\chi^*)$, while the one in $s=-a_\chi=0$ (i.e. $n=0$) is double and its contribution is
\[
x(1 + \beta - \log x),
\]
where $-\frac{L'}{L}(\epsilon,\chi^*) =: -\frac{1}{\epsilon} + \beta + O(\epsilon)$ when $\chi^*$ is not
trivial, and is simple with contribution equal to $-\frac{L'}{L}(0,\chi^*)x$ when $\chi^*$ is trivial.
\end{proof}

%
%
%
%
%

\section{General setting and partial results}
\label{sec:A3}
Let $q\in\N$ and $a\in\Z$ with $(a,q)=1$. For any sequence $f=\{f_\chi\}$ of objects depending on the
character $\chi$ modulo $q$ let
\[
M_{a,q}f_\chi := \frac{1}{\eul(q)}\sum_{\chi\in\widehat{(\Z/q\Z)^*}} \overline{\chi}(a) f_\chi.
\]
The operator $M_{a,q}$ selects for the integers which are congruent to $a$ modulo $q$. Notice that
$M_{a,q}$ is akin to the mean value, since if $|f_\chi|\leq M$ for every character, then
$|M_{a,q}f_\chi|\leq M$.\\
Moreover, for any function $f\colon \R\to\C$ we let
\[
\Delta_{2,h}f := f(x+h)-2f(x)+f(x-h).
\]
The operator $\Delta_{2,h}$ will select the integers which are in the interval $(x-h,x+h)$.\\
Notably, the operators $M_{a,q}$ and $\Delta_{2,h}$ commute and
\[
M_{a,q}\Delta_{2,h}\psi_\chi^{(1)}(x)
= \sum_{n=a\Pmod{q}} \Lambda(n)K(x-n;h)
\]
where $K(u;h):=\max\{h-|u|,0\}$, so that it is supported in $|u|\leq h$, is positive in the open set, and
has a unique maximum at $u=0$ with $K(0;h)=h$. The theorem follows from this basic equality by estimating,
in the standard way, the function appearing on the left hand side. Since Lemma~\ref{lem:A1} is valid only
for $\chi^*$, we firstly need to connect $\psi_{\chi}^{(1)}$ with $\psi_{\chi^*}^{(1)}$. To this end, we let
\[
B(\chi,x) := \psi_\chi^{(1)}(x) - \psi_{\chi^*}^{(1)}(x)
\]
and prove the following lemma.
\begin{lemma}\label{lem:A2}
Assume $x\geq 1$ and $0< h\leq x$. Then
\[
|M_{a,q}\Delta_{2,h} B(\chi,x)| \leq \omega(q)h\log(2x).
\]
\end{lemma}
\begin{proof}
We will prove that
\begin{equation}\label{eq:A4}
|\Delta_{2,h} B(\chi,x)| \leq \omega(q)h\log(2x),
\end{equation}
and the claim will immediately follow by the mean value property of $M_{a,q}$. By~\eqref{eq:A1} we have that
\[
B(\chi,x) = \sum_{n\leq x} (\chi(n)-\chi^*(n))\Lambda(n)(x-n).
\]
Thus, only those integers that are coprime to $q_\chi$ and not $q$ will be counted, giving
\[
B(\chi,x) = -\sum_{\substack{n\leq x\\ (n,q)>1 \\ (n,q_\chi)=1}} \chi^*(n)\Lambda(n)(x-n).
\]
It follows that
\[
\Delta_{2,h}B(\chi,x) = -\sum_{\substack{(n,q)>1 \\ (n,q_\chi)=1}} \chi^*(n)\Lambda(n)K(x-n;h),
\]
and therefore
\[
|\Delta_{2,h}B(\chi,x)| \leq \sum_{\substack{(n,q)>1 \\ (n,q_\chi)=1}} \Lambda(n)K(x-n;h).
\]
Recalling the definition of $\Lambda$ and removing the restriction $(n,q_\chi)=1$, we get
\begin{equation}\label{eq:A5}
|\Delta_{2,h}B(\chi,x)|
\leq \sum_{p|q} \log p \sum_{k\geq 1} K(x-p^k;h).
\end{equation}
The inner sum is trivially bounded by
\[
h\sum_{\substack{p^k< x+h\\1\leq k}} 1
\leq h \intpart{\frac{\log(x+h)}{\log p}}.
\]
Finally,~\eqref{eq:A5} gives
\[
|\Delta_{2,h}B(\chi,x)|
\leq h\sum_{p|q} \log(x+h)
= \omega(q)h\log(x+h)
\]
which is~\eqref{eq:A4} under the restriction $0<h\leq x$.
\end{proof}

\begin{lemma}\label{lem:A3}
Assume $q\geq 3$, $x\geq 100$ and $0 < h\leq \frac{5}{6}x$. Then
\begin{multline*}
M_{a,q}\Delta_{2,h}\psi_{\chi^*}^{(1)}(x)
= \frac{h^2}{\eul(q)}\\
  + \theta\Big[
           \Big|M_{a,q}\Delta_{2,h}\sum_{\rho\in Z_{\chi^*}} \frac{x^{\rho+1}}{\rho(\rho+1)}\Big|
           + 1.7\Big|\delta_{\pm 1[q]}(a) - \frac{2}{\eul(q)}\Big|\frac{h^2}{x}
           + \delta_{\pm 1[q]}(a)\frac{6h^2}{x(x-1)}
          \Big]
\end{multline*}
%
for some $\theta=\theta(a,q,x,h)\in [-1,1]$, where $\delta_{\pm 1[q]}(a)=1$ when $a=\pm 1\pmod{q}$ and
$0$ otherwise.
\end{lemma}
The value $\frac{5}{6}$ in the upper bound $h\leq \frac{5}{6}x$ could be changed in a quite large
interval without affecting the final result. However, in order to bound the secondary terms as $h^2/x$
and $h^2/x^2$ respectively, it is essential to have an upper bound for $h/x$ strictly smaller than $1$.
\begin{proof}
We apply the operator $M_{a,q}\Delta_{2,h}$ to~\eqref{eq:A3}. We notice that $\Delta_{2,h}x^{j}=0$ for
$j=0$, $1$, and in general
\[
\Delta_{2,h}f(x) = \int_0^h (h-u)(f''(x+u)+f''(x-u))\d u
\]
for every $C^2$ function. Thus
\begin{align}
M_{a,q}\Delta_{2,h} \Big(\frac{x^2}{2}\delta_{\chi^*=\tc}\Big)
&= M_{a,q}(\delta_{\chi^*=\tc})\cdot\Delta_{2,h}\Big(\frac{x^2}{2}\Big)
= \frac{h^2}{\eul(q)},                                                                \label{eq:A6}\\
M_{a,q}\Delta_{2,h} (- x r_{\chi^*} + r'_{\chi^*})
&= 0,                                                                                 \label{eq:A7}
\end{align}
and we have still to bound $M_{a,q}\Delta_{2,h}R_{\chi^*}^{(1)}(x)$. This is the sum of three terms:
\begin{gather*}
M_{a,q}\Delta_{2,h} \big(a_\chi\log x\big),
\qquad\qquad
M_{a,q}\Delta_{2,h} \big((1-a_\chi)\delta_{\chi^*\neq \tc}\,x\log x\big),                        \\
M_{a,q}\Delta_{2,h} \Big(\sum_{n=1}^\infty \frac{x^{1-2n-a_\chi}}{(2n+a_\chi)(2n+a_\chi-1)}\Big).
\end{gather*}
Since the set of even characters is a subgroup of $\widehat{(\Z/q\Z)^*}$ of index two, we have
\begin{align*}
M_{a,q}(1-a_\chi)
&= \frac{1}{\eul(q)}\sum_{\substack{\chi\in \widehat{(\Z/q\Z)^*}\\ \chi \even}}\overline{\chi(a)}
= \frac{1}{2}\delta_{\pm 1[q]}(a),                                                                \\
M_{a,q}((1-a_\chi)\delta_{\chi^*\neq \tc})
&= \frac{1}{\eul(q)}\Big(\sum_{\substack{\chi\in \widehat{(\Z/q\Z)^*}\\ \chi \even}}\overline{\chi(a)}-1\Big)
= \frac{1}{2}\delta_{\pm 1[q]}(a) - \frac{1}{\eul(q)} ,                                           \\
M_{a,q}(a_\chi)
&= \frac{1}{\eul(q)}\sum_{\substack{\chi\in \widehat{(\Z/q\Z)^*}\\ \chi \odd}}\overline{\chi(a)}
= \frac{1}{2}\delta_{\pm 1[q]}(a)\eta(a),
\end{align*}
where $\eta$ is any odd character modulo $q$. Moreover,
\begin{align*}
|\Delta_{2,h} \log x|
&=    \int_0^h (h-u)\Big(\frac{1}{(x+u)^2} + \frac{1}{(x-u)^2}\Big)\d u
=    \int_0^h \frac{2(h-u)(x^2+u^2)}{(x^2-u^2)^2}\d u\\
&\leq \frac{2}{(x^2-h^2)^2}\int_0^h (h-u)(x^2+u^2)\d u
=    \frac{h^2(x^2 + \frac{h^2}{6})}{(x^2-h^2)^2}
\leq  12 \frac{h^2}{x^2},
\end{align*}
%
where for the last inequality we have used the assumption $0<h\leq 5x/6$. Thus we have that
\begin{equation}\label{eq:A8}
|M_{a,q}\Delta_{2,h} \big(a_\chi\log x\big)|
\leq 6\frac{h^2}{x^2}\delta_{\pm 1[q]}(a).
\end{equation}
\smallskip
Similarly, for $0\leq h\leq 5x/6$ it follows that
\[
|\Delta_{2,h}x\log x|
\leq \frac{xh^2}{x^2-h^2}
\leq 3.4\frac{h^2}{x},
\]
%
%
and so
\begin{equation}\label{eq:A9}
|M_{a,q}\Delta_{2,h}\big((1-a_\chi)\delta_{\chi^*\neq \tc}\,x\log x\big)|
\leq 1.7\Big|\delta_{\pm 1[q]}(a) - \frac{2}{\eul(q)}\Big|\frac{h^2}{x}.
\end{equation}
\smallskip
Lastly,
\begin{align*}
\Delta_{2,h} \sum_{n=1}^\infty &\frac{x^{1-2n-a_\chi}}{(2n+a_\chi)(2n+a_\chi-1)}
 =    \int_0^h (h-u)\sum_{n=1}^\infty \Big((x+u)^{-1-2n-a_\chi} + (x-u)^{-1-2n-a_\chi} \Big)\d u\\
&=    \int_0^h (h-u)\Big(\frac{(x+u)^{-1-a_\chi}}{(x+u)^2-1} +
\frac{(x-u)^{-1-a_\chi}}{(x-u)^2-1}\Big)\d u.
\intertext{Using $h\leq 5x/6$ (and taking $u=vh$ with the fact that the function increases in $h$), we
have that the above expression is bounded above by}
& \frac{h^2}{x^{3+a_\chi}}\int_0^1 (1-v)\Big(\frac{(1+5v/6)^{-1-a_\chi}}{(1+5v/6)^2-x^{-2}} + \frac{(1-5v/6)^{-1-a_\chi}}{(1-5v/6)^2-x^{-2}}\Big)\d v\\
&= \frac{6h^2}{5x^{3+a_\chi}}\int_0^{5/6}\Big(1-\frac{6}{5}w\Big)\Big(\frac{(1+w)^{-1-a_\chi}}{(1+w)^2-x^{-2}} + \frac{(1-w)^{-1-a_\chi}}{(1-w)^2-x^{-2}}\Big)\d w.
\intertext{Since $x\geq 100$, this is bounded above by}
& \frac{6h^2}{5x^{3+a_\chi}}\int_0^{5/6}\Big(1-\frac{6}{5}w\Big)\Big(\frac{(1+w)^{-1-a_\chi}}{(1+w)^2-100^{-2}} + \frac{(1-w)^{-1-a_\chi}}{(1-w)^2-100^{-2}}\Big)\d w
 \leq 12\frac{h^2}{x^{3+a_\chi}}.
\end{align*}
%
%
Thus, we have that
\begin{equation}\label{eq:A10}
\Big|M_{a,q}\Delta_{2,h} \Big(\sum_{n=1}^\infty \frac{x^{1-2n-a_\chi}}{(2n+a_\chi)(2n+a_\chi-1)}\Big)\Big|
\leq 6\frac{h^2}{x^3}\Big(1 + \frac{1}{x}\Big)\delta_{\pm 1[q]}(a),
\end{equation}
and now the claim follows from~\eqref{eq:A3} and~(\ref{eq:A6}--\ref{eq:A10}).
\end{proof}
We split the sum on zeros as
\[
\sum_{\rho\in Z_{\chi^*}} \frac{x^{\rho+1}}{\rho(\rho+1)}
=: \Sigma_{\chi^*,1} + \Sigma_{\chi^*,2},
\]
with $\Sigma_{\chi^*,1}$ and $\Sigma_{\chi^*,2}$ representing the sums on zeros with $|\Imm(\rho)|\leq T$
and $|\Imm(\rho)|> T$, respectively, for a convenient parameter $T>0$. The next lemma provides a bound
for $\Sigma_{\chi^*,2}$.
\begin{lemma}\label{lem:A4}
Assume GRH, $q\geq 3$, $0\leq h\leq x$ and $T\geq 16$. Then
\[
\big|M_{a,q}\Delta_{2,h}\Sigma_{\chi^*,2}\big|
\leq \frac{4}{\pi}\Big(x^{3/2}+\frac{h^2}{4\sqrt{x}}\Big)\Big(1 + \frac{2.89}{T}\Big)\frac{\log(qT)}{T}.
\]
\end{lemma}
\begin{proof}
For $z\in[0,1]$ by double squaring we get $(1+z)^{3/2} + (1-z)^{3/2} \leq 2 + z^2$ which implies that
$(x+h)^{3/2}+2x^{3/2}+(x-h)^{3/2}\leq 4x^{3/2} + h^2/\sqrt{x}$ for $0\leq h\leq x$.
Thus, GRH gives us that
\[
\big|\Delta_{2,h}\Sigma_{\chi^*,2}\big|
\leq 4\Big(x^{3/2}+\frac{h^2}{4\sqrt{x}}\Big) \sum_{\substack{\rho\in Z_{\chi^*}\\ |\Imm(\rho)|> T}} \frac{1}{|\rho(\rho+1)|},
\]
so that
\[
\big|M_{a,q}\Delta_{2,h}\Sigma_{\chi^*,2}\big|
\leq \frac{4}{\eul(q)}\Big(x^{3/2}+\frac{h^2}{4\sqrt{x}}\Big) \sum_{\chi\in\widehat{(\Z/q\Z)^*}}\sum_{\substack{\rho\in Z_{\chi^*}\\ |\Imm(\rho)|> T}} \frac{1}{|\rho(\rho+1)|}.
\]
Each inner sum on zeros could be estimated by partial summation using the known formulas for the number
of zeros of each $L(s,\chi)$ (see Trudgian~\cite{TrudgianIII}), but we can reduce the error term by
connecting the sum with a similar sum for a Dedekind zeta function. In fact one has the factorization
$\zeta_{\Q[q]}(s) = \prod_{\chi\in\widehat{(\Z/q\Z)^*}}L(s,\chi^*)$, where $\Q[q]$ is the cyclotomic
field of $q$-roots of unity (see~\cite[Th.~4.3]{Washington1}), and thus
\begin{equation}\label{eq:A11}
|M_{a,q}\Delta_{2,h}\Sigma_{\chi^*,2}|
\leq \frac{4}{\eul(q)}\Big(x^{3/2}+\frac{h^2}{4\sqrt{x}}\Big) \sum_{\substack{\rho\in Z_q\\ |\Imm(\rho)|> T}} \frac{1}{|\rho(\rho+1)|},
\end{equation}
where $Z_q$ is the multiset of zeros of $\zeta_{\Q[q]}$. This sum has already been estimated
in~\cite[Eq.~(3.7)]{GrenieMolteni3} for a generic number field $\K$, the result being that
\[
\sum_{|\gamma|\geq T}\frac{\pi}{|\rho|^2}
\leq \Big(1 + \frac{2.89}{T}\Big)\frac{W_\K(T)}{T}
    +\Big(1 + \frac{18.61}{T}\Big)\frac{n_\K}{T}
    +\frac{17.31}{T^2}
\]
for all $T \geq 5$ where $W_\K(T) := \log\disc_\K+n_\K\log(T/2\pi)$, $\disc_\K$ is the absolute value of
the discriminant of $\K$ and $n_\K$ its degree. For $\K=\Q[q]$, one has that $\log\disc_\K=\eul(q)\log q
-\eul(q)\sum_{p|q}\frac{\log p}{p-1}$ (see~\cite[Proposition~2.7]{Washington1}) and $n_\K=\eul(q)$, thus
this formula becomes
\begin{equation}\label{eq:A11bis}
\sum_{|\gamma|\geq T}\frac{\pi}{|\rho|^2}
\leq \Big(1 + \frac{2.89}{T}\Big)\frac{\eul(q)}{T}\Big(\log q - \sum_{p|q}\frac{\log p}{p-1}+\log\Big(\frac{T}{2\pi}\Big)\Big)
    + \Big(1 + \frac{18.61}{T}\Big)\frac{\eul(q)}{T}
    + \frac{17.31}{T^2}
\end{equation}
for all $T \geq 5$. We simplify this to
\begin{equation}\label{eq:A12}
\sum_{|\gamma|\geq T}\frac{\pi}{|\rho|^2}
\leq \Big(1 + \frac{2.89}{T}\Big)\frac{\eul(q)}{T}\log(qT)
\end{equation}
for all $T \geq 16$.
Indeed, \eqref{eq:A11bis} shows that \eqref{eq:A12} holds as soon as
\[
1+\frac{18.61-2.89\log2\pi}{T} + \frac{17.31}{\eul(q)T}
\leq \Big(1+\frac{2.89}{T}\Big)\sum_{p|q}\frac{\log p}{p-1}
    +\log 2\pi,
\]
which is implied for $T\geq 16$ by
\begin{equation}\label{eq:A12bis}
1+\frac{18.61-2.89\log2\pi}{16} + \frac{17.31}{16\eul(q)}
\leq \sum_{p|q}\frac{\log p}{p-1}
    +\log2\pi.
\end{equation}
By inspection we test that this inequality holds for each $q=3,\ldots,1000$.
On the other hand, if $q\geq 1000$, then using the multiplicativity of $\eul(q)/q^{3/4}$ one can prove
easily that $\eul(q)\geq q^{3/4} > 170$. Thus~\eqref{eq:A12bis} still holds
because $1+\frac{18.61-2.89\log2\pi}{16} + \frac{17.31}{16\cdot 170} \leq \log2\pi$.\\
%
The proof concludes combining~\eqref{eq:A11} and~\eqref{eq:A12}.
\end{proof}
Collecting the results in Lemmas~\ref{lem:A1},~\ref{lem:A2},~\ref{lem:A3} and~\ref{lem:A4} we get
\begin{multline*}
\sum_{n=a\Pmod{q}} \Lambda(n) K(x-n;h)
\geq \frac{h^2}{\eul(q)}
  - |M_{a,q}\Delta_{2,h}\Sigma_{\chi^*,1}|
  - \frac{4}{\pi}\Big(x^{3/2}+\frac{h^2}{4\sqrt{x}}\Big)\Big(1 + \frac{2.89}{T}\Big)\frac{\log(qT)}{T}  \\
  - \omega(q)h\log(2x)
  - \Big[1.7\Big|\delta_{\pm 1[q]}(a) - \frac{2}{\eul(q)}\Big|\frac{h^2}{x}
         + \delta_{\pm 1[q]}(a)\frac{6h^2}{x(x-1)}
    \Big].
\end{multline*}
We simplify it by noticing that
\[
  1.7\Big|\delta_{\pm 1[q]}(a) - \frac{2}{\eul(q)}\Big|\frac{h^2}{x}
         + \delta_{\pm 1[q]}(a)\frac{6h^2}{x(x-1)}
  \leq 1.7\frac{h^2}{x}
\]
when $q\geq 3$ (thus $\eul(q)\geq 2$) and $x\geq 2\eul(q)+1$. In this way we deduce that
%
%
\begin{align*}
\sum_{n=a\Pmod{q}} \Lambda(n) K(x-n;h)
\geq& \frac{h^2}{\eul(q)}
  - |M_{a,q}\Delta_{2,h}\Sigma_{\chi^*,1}|
  - 4\Big(x^{3/2}+\frac{h^2}{4\sqrt{x}}\Big)\Big(1 + \frac{2.89}{T}\Big)\frac{\log(qT)}{\pi T} \\
  &- \omega(q)h\log(2x)
  - 1.7\frac{h^2}{x}.
\end{align*}
Now we remove the contribution of prime powers. We get
\begin{align*}
\sum_{n=a\Pmod{q}} \Lambda(n) K(x-n;h)
\leq h\sum_{\substack{|n-x|<h\\ n=a\Pmod{q}}} \Lambda(n)
= h\Big(\sum_{\substack{|p-x|<h\\ p=a\Pmod{q}}} \log p + \sum_{2\leq k}\sum_{\substack{|p^k-x|<h\\ p^k=a\Pmod{q}}} \log p\Big)
\end{align*}
and removing the arithmetical condition one gets
\begin{align*}
\sum_{2\leq k}\sum_{\smash[b]{\substack{|p^k-x|<h\\ p^k=a\Pmod{q}}}} \log p
&\leq\sum_{2\leq k}\sum_{|p^k-x|<h} \log p
\leq [\psi(x+h)-\vartheta(x+h)] - [\psi(x-h)-\vartheta(x-h)]     \\
&\leq (1+10^{-6})\sqrt{x+h} + 3\sqrt[3]{x+h}
    - 0.998 684\sqrt{x-h}
\end{align*}
for every $x\geq 121$ (see~\cite[Cor.~2]{PlattTrudgian} and \cite[Th.~6]{RosserSchoenfeld2}). Assuming
that $h\leq 5x/6$ (as we have done for Lemma~\ref{lem:A3}) we have
\begin{equation}\label{eq:A13}
\sum_{2\leq k}\sum_{\substack{|p^k-x|<h\\ p^k=a\Pmod{q}}} \log p
\leq 0.95\sqrt{x} + 3.7\sqrt[3]{x}.
\end{equation}
%
%
Note that the Brun--Titchmarsh theorem for primes in arithmetic progressions (eventually in intervals --
see~\cite{MontgomeryVaughan2}) produces a much better bound, but only when $x$ and $q$ are much
larger than what we need to prove our theorem. As a consequence we have decided not to use this tool.\\
To summarise so far, we have proved that for $q\geq 3$, $x\geq \max(121,2\eul(q)+1)$ and $h\leq
\frac{5}{6}x$ one has
\begin{equation}
\begin{aligned}
\sum_{\smash[b]{\substack{|p-x|<h\\ p=a\Pmod{q}}}} \log p
\geq& \frac{h}{\eul(q)}
  - \frac{1}{h}|M_{a,q}\Delta_{2,h}\Sigma_{\chi^*,1}|
  - 4\Big(x^{3/2}+\frac{h^2}{4\sqrt{x}}\Big)\Big(1 + \frac{2.89}{T}\Big)\frac{\log(qT)}{\pi hT} \\
 &- (0.95\sqrt{x} + 3.7\sqrt[3]{x})
  - \omega(q)\log(2x)
  - 1.7\frac{h}{x}.                                                                             \label{eq:A14}
\end{aligned}
\end{equation}
In Section~\ref{sec:A4} we provide an upper bound for $M_{a,q}\Delta_{2,h}\Sigma_{\chi^*,1}$. In this way
we will be able to prove the theorems in Sections~\ref{sec:A5} and~\ref{sec:A6}.

\section{Bound for $M_{a,q}\Delta_{2,h}\Sigma_{\chi^*,1}$}
\label{sec:A4}

\begin{lemma}\label{lem:A5}
Let $y>0$. Then for some $\theta\in [-1,1]$ we have that
\begin{align*}
\int_{0}^{y} \frac{\sin^2 t}{t^2} \d t
= \frac{\pi}{2} - \frac{1}{2y} + \frac{\theta}{4y^2}.
\end{align*}
\end{lemma}
The claim with $\theta\in[-2,2]$ has a very simple proof. We optimize the result by proving the stronger
bound $\theta\in[-1,1]$.
\begin{proof}
We note that
\begin{align*}
\int_{0}^{y} \frac{\sin^2 t}{t^2} \d t
&= \frac{\pi}{2} - \int_{y}^{\infty} \frac{\sin^2 t}{t^2} \d t
= \frac{\pi}{2} - \int_{y}^{\infty} \frac{1-\cos(2t)}{2t^2} \d t
= \frac{\pi}{2} - \frac{1}{2y} + \int_{y}^{\infty} \frac{\cos(2t)}{2t^2} \d t.
\end{align*}
Therefore the claim states that $|f(x)|\leq 1$ when $x>0$, where $f(x):= x^2\int_x^{+\infty}\frac{\cos
v}{v^2}\d v$.
%
%
We prove this statement in two steps.
\begin{enumerate}
\item[Step 1)] The claim holds in $[0,6]$.\\
We notice that $f$ is the unique bounded solution in $(0,+\infty)$ of the ODE
$y'=\frac{2}{x}(y-\frac{x}{2}\cos x)$. We can use this equation to trace the graph of $f(x)$ in $[0,6]$.
The extremal points of $f$ solve $f(x)= \frac{x}{2}\cos x=: g(x)$, so $|f(x)|<1$ when $0<x\leq 2.4$,
since $\frac{x}{2}|\cos x|<1$ here.
%
%
Moreover, $f(5)=0.896\ldots$ so that $f'(5)> \frac{2}{5}(0.896-\frac{5}{2}\cos(5))>0$, and
$f(5.1)=0.899\ldots$ so that $f'(5)< \frac{2}{5.1}(0.9-\frac{5.1}{2}\cos(5.1))<0$,
%
%
thus $f$ is increasing for $x\in [2.4,5]$ (by the differential equation) and smaller than $1$ here
(because $f(5)<1$). Moreover there is a maximum for $f$ in $[5,5.1]$, and since $g(x)<0.97$ here, we
conclude that $|f(x)|<1$ in $[0,5.1]$. Moreover, $g(x)$ increases for $x\in[5.1,6]$ and $f(6)=0.50\ldots
< g(6)$,
%
%
thus $f(x)$ decreases here, and the value of $f(6)$ completes the proof of this step.
\item[Step 2)] The claim holds for $x\geq 6$.\\
Four integrations by parts give
\[
f(x) = -\sin x + 2\frac{\cos x}{x} + 6\frac{\sin x}{x^2} - 24\frac{\cos x}{x^3} + 120x^2\int_x^{+\infty}\frac{\cos v}{v^6}\d v
\]
so that
\[
|f(x)| \leq \Big|\Big(1 - \frac{6}{x^2}\Big)\sin x - \frac{2}{x}\Big(1 - \frac{12}{x^2}\Big)\cos x\Big| + \frac{24}{x^3}.
\]
We prove that this function is lower than $1$ for $x\geq 6$.
%
%
Multiplying by $x^2$, we have to prove that
\[
-x^2 + \frac{24}{x} < (x^2 - 6)\sin x - \frac{2}{x}(x^2 - 12)\cos x < x^2 - \frac{24}{x}.
\]
The first inequality is evident when $\cos x\leq 0$, and the second when $\cos x\geq 0$, respectively
(because we are assuming $x\geq 6$). Assuming $\cos x >0$ for the first one, and $\cos x <0$ for the
second one, both remaining inequalities are implied by the stronger bound:
\[
(x^2 - 6)\Big|\sin x - \frac{2}{x}\cos x\Big| < x^2 - \frac{24}{x}.
\]
Since $|\sin x -\alpha \cos x|\leq \sqrt{1+\alpha^2}\leq 1+\tfrac{\alpha^2}{2}$ (the first inequality by
elementary trigonometry, the second by convexity), it is sufficient to prove that
\[
(x^2 - 6)\Big(1 + \frac{2}{x^2}\Big) < x^2 - \frac{24}{x},
\]
which in fact holds for $x\geq 6$.
\end{enumerate}
\end{proof}

\begin{lemma}\label{lem:A6}
Let $0\leq h < x$. Then for every $\gamma\in\R$ there exists $\theta\in\C$ with $|\theta|\leq 1$ such
that
\[
\Big(1+\frac{h}{x}\Big)^{\frac{3}{2}+i\gamma} -2  + \Big(1-\frac{h}{x}\Big)^{\frac{3}{2}+i\gamma}
= -4\sin^2\Big(\frac{\gamma h}{2x}\Big) + \theta (2|\gamma|+1)\frac{h^2}{x^2}.
\]
\end{lemma}
\begin{proof}
The proof is straightforward and follows from the Taylor expansion of $\log(1+u)$ and some elementary
inequalities.
\end{proof}

The definitions of $\Delta_{2,h}$ and $\Sigma_{\chi^*,1}$ show that
\[
\Delta_{2,h}\Sigma_{\chi^*,1}
 = \sum_{\substack{\rho\in Z_{\chi^*}\\ |\Imm(\rho)|\leq T}} \frac{(x+h)^{\rho+1}-2x^{\rho+1}+(x-h)^{\rho+1}}{\rho(\rho+1)},
\]
so that by Lemma~\ref{lem:A6} we deduce that
\begin{align*}
|\Delta_{2,h}\Sigma_{\chi^*,1}|
&\leq 4x^{3/2}\sum_{\substack{\tfrac{1}{2}+i\gamma\in Z_{\chi^*}\\ |\gamma|\leq T}}
                 \frac{\sin^2\big(\frac{\gamma h}{2x}\big)}{|(\tfrac{1}{2}+i\gamma)(\tfrac{3}{2}+i\gamma)|}
      + \frac{h^2}{\sqrt{x}}\sum_{\substack{\tfrac{1}{2}+i\gamma\in Z_{\chi^*}\\ |\gamma|\leq T}}
                 \frac{2|\gamma|+1}{|(\tfrac{1}{2}+i\gamma)(\tfrac{3}{2}+i\gamma)|}.
\end{align*}

As we have done for $\Sigma_{\chi^*,2}$ we use the factorization of the Dedekind zeta function $\zeta_\K$
of the cyclotomic field $\K:=\Q[q]$ of $q$-th roots of unity as products of $L(s,\chi^*)$; in this way we
deal with all zeros in $\cup_{\chi\in\widehat{(\Z/q\Z)^*}} Z_{\chi^*}$ as a unique step. This does not affect the
main part of the theorem, but reduces the size of the secondary terms, and makes the ranges for $q$ and
$x$ wider in the theorem.
\begin{align}
\eul(q)|M_{a,q}\Delta_{2,h}\Sigma_{\chi^*,1}|
&\leq \sum_{\chi\in\widehat{(\Z/q\Z)^*}}
      \Big[4x^{3/2}\sum_{\substack{\tfrac{1}{2}+i\gamma\in Z_{\chi^*}\\ |\gamma|\leq T}}
                      \frac{\sin^2\big(\frac{\gamma h}{2x}\big)}{|(\tfrac{1}{2}+i\gamma)(\tfrac{3}{2}+i\gamma)|}
           + \frac{h^2}{\sqrt{x}}\sum_{\substack{\tfrac{1}{2}+i\gamma\in Z_{\chi^*}\\ |\gamma|\leq T}}
                      \frac{2|\gamma|+1}{|(\tfrac{1}{2}+i\gamma)(\tfrac{3}{2}+i\gamma)|}
      \Big]                                                                                              \notag\\
&=  4x^{3/2}\sum_{\substack{\tfrac{1}{2}+i\gamma\in Z_{q}\\ |\gamma|\leq T}}
               \frac{\sin^2\big(\frac{\gamma h}{2x}\big)}{|(\tfrac{1}{2}+i\gamma)(\tfrac{3}{2}+i\gamma)|}
    + \frac{h^2}{\sqrt{x}}\sum_{\substack{\tfrac{1}{2}+i\gamma\in Z_{q}\\ |\gamma|\leq T}}
               \frac{2|\gamma|+1}{|(\tfrac{1}{2}+i\gamma)(\tfrac{3}{2}+i\gamma)|}.                      \label{eq:A15}
\end{align}
We deduce a bound for the second sum from two computations already made by the second and third author
for Dedekind zeta functions.
\begin{lemma}\label{lem:A7}
Assume GRH and let $T\geq 20$. Then
\[
\frac{1}{\eul(q)}\sum_{\substack{\tfrac{1}{2}+i\gamma\in Z_{q}\\ |\gamma|\leq T}}
                    \frac{2|\gamma|+1}{|(\tfrac{1}{2}+i\gamma)(\tfrac{3}{2}+i\gamma)|}
\leq \frac{1}{\pi}\log(q^2T)\log T + 1.93\log q - 4.35 + \frac{21.67}{\eul(q)}.
\]
%
%
\end{lemma}
\begin{proof}
In~\cite[Eq.~(3.8)]{GrenieMolteni3} it is proved that
\[
\sum_{\tfrac{1}{2}+i\gamma\in Z_{q}} \frac{\pi}{|\tfrac{1}{2}+i\gamma|}
\leq \Big(\log\Big(\frac{T}{2\pi}\Big) + 4.01\Big)\log \dq
     + \Big(\frac{1}{2}\log^2\Big(\frac{T}{2\pi}\Big) - 1.41\Big)\eul(q)
     + 25.57;
\]
and in~\cite[Lemma~4.1]{GrenieMolteni2} that
\[
\sum_{\tfrac{1}{2}+i\gamma\in Z_{q}} \frac{1}{|(\tfrac{1}{2}+i\gamma)(\tfrac{3}{2}+i\gamma)|}
\leq 0.54\log \dq - 1.03 \eul(q) + 5.39
\]
(both for $T\geq 5$). Thus,
\begin{align*}
&\sum_{\substack{\tfrac{1}{2}+i\gamma\in Z_{q}\\ |\gamma|\leq T}} \frac{2|\gamma|+1}{|(\tfrac{1}{2}+i\gamma)(\tfrac{3}{2}+i\gamma)|}
 \leq \sum_{\substack{\tfrac{1}{2}+i\gamma\in Z_{q}\\ |\gamma|\leq T}} \frac{2}{|\tfrac{1}{2}+i\gamma|}
+ \sum_{\tfrac{1}{2}+i\gamma\in Z_{q}} \frac{1}{|(\tfrac{1}{2}+i\gamma)(\tfrac{3}{2}+i\gamma)|} \\
&\leq \frac{2}{\pi}
      \Big[\Big(\log\Big(\frac{T}{2\pi}\Big) + 4.01\Big)\log \dq
           + \Big(\frac{1}{2}\log^2\Big(\frac{T}{2\pi}\Big) - 1.41\Big)\eul(q)
           + 25.57
      \Big]
    + 0.54\log \dq - 1.03 \eul(q) + 5.39
\end{align*}
and recalling that $\log\dq=\eul(q)\log q -\eul(q)\sum_{p|q}\frac{\log p}{p-1} \leq \eul(q)\log q$,
we get
\begin{align*}
\frac{1}{\eul(q)}
&\sum_{\substack{\tfrac{1}{2}+i\gamma\in Z_{q}\\ |\gamma|\leq T}}
 \frac{2|\gamma|+1}{|(\tfrac{1}{2}+i\gamma)(\tfrac{3}{2}+i\gamma)|}
 \leq \Big(\frac{2}{\pi}\log\Big(\frac{T}{2\pi}\Big) + 3.1\Big)\log q
           + \Big(\frac{1}{\pi}\log^2\Big(\frac{T}{2\pi}\Big) - 1.927\Big)
           + \frac{21.67}{\eul(q)}                                                           \\
%
%
&= \frac{1}{\pi}\log(q^2T)\log T
           + \Big(3.1-\frac{2}{\pi}\log(2\pi)\Big)\log q
           - \frac{2}{\pi}\log T\log(2\pi)
           + \frac{1}{\pi}\log^2(2\pi)
           - 1.927
           + \frac{21.67}{\eul(q)}.
\end{align*}
The claim follows by recalling that we are assuming $T\geq 20$ so that the contribution of all secondary
terms is $-4.35$, at most.
%
%
\end{proof}

\begin{lemma}\label{lem:A8}
Assume GRH and Let $\K$ be any number field. Then
\begin{align*}
\sum_{|\gamma|\leq 5}\frac{\gamma^2}{|(\tfrac{1}{2}+i\gamma)(\tfrac{3}{2}+i\gamma)|}
&\leq 1.5\log\disc_\K + 1.651n_\K - 1.577.
\end{align*}
\end{lemma}
\begin{proof}
We apply the same technique we have already used for Lemmas~3.1 and 4.1 in~\cite{GrenieMolteni2} and for
Lemma~3.1 in~\cite{GrenieMolteni3}, stemming from the remark that the function $f_\K(s) :=
\sum_\rho\Ree\big(\frac{2}{s-\rho}\big)$ can be exactly computed via the alternative representation
\begin{equation}\label{eq:A16}
f_\K(s)
= 2\Ree\frac{\zeta'_\K}{\zeta_\K}(s)
    +\log\frac{\disc_\K}{\pi^{n_\K}}
    +\Ree\Big(\frac{2}{s}
             +\frac{2}{s-1}
        \Big)
    +(r_1+r_2)\Ree\frac{\Gamma'}{\Gamma}\Big(\frac{s}{2}\Big)
    +r_2\Ree\frac{\Gamma'}{\Gamma}\Big(\frac{s+1}{2}\Big).
\end{equation}
Let
\[
f(s,\gamma):=\frac{4(2s-1)}{(2s-1)^2+4\gamma^2},
\]
so that $f_\K(s) = \sum_{\gamma} f(s,\gamma)$, and let
\[
g(\gamma) :=
\begin{cases}
\frac{\gamma^2}{((\frac{1}{4}+\gamma^2)(\frac{9}{4}+\gamma^2))^{1/2}} & \text{if }|\gamma|\leq 5 \\
0                                                                     & \text{otherwise}.
\end{cases}
\]
so that $\sum_{|\gamma|\leq 5} \frac{\gamma^2}{|(\tfrac{1}{2}+i\gamma)(\tfrac{3}{2}+i\gamma)|} =
\sum_{\gamma} g(\gamma)$. We look for a finite linear combination of $f(s,\gamma)$ at suitable points
$s_j$ such that
\begin{equation}\label{eq:A17}
g(\gamma)
\leq F(\gamma) := \sum_j a_j f(s_j,\gamma)
\end{equation}
for all $\gamma \in \R$ so that
\begin{equation}\label{eq:A18}
\sum_{|\gamma|\leq 5}\frac{\gamma^2}{|(\tfrac{1}{2}+i\gamma)(\tfrac{3}{2}+i\gamma)|}
\leq \sum_j a_j f_\K(s_j).
\end{equation}
Once~\eqref{eq:A18} is proved, we recover a bound for the sum on zeros by recalling the
identity~\eqref{eq:A16}. According to this approach, the final coefficient of $\log\disc_\K$ will be the
sum of all $a_j$, and thus we are interested in the linear combinations for which this sum is as small as
possible. We set $s_j = 3/4 + j/2$ with $j=1$, \ldots, $2\kappa + 3$ for a suitable integer $\kappa$.
Let $\Upsilon \subset (0,+\infty)$ be a set with $\kappa$ numbers. We require:
\begin{enumerate}
\item $F(\gamma)=g(\gamma)$ for all $\gamma\in\Upsilon\cup\{0,5\}$,
\item $F'(\gamma)=g'(\gamma)$ for all $\gamma\in\Upsilon$,
\item $\lim_{\gamma\to\infty}\gamma^2 F(\gamma)=\lim_{\gamma\to\infty}\gamma^2 g(\gamma)=0$.
\end{enumerate}
This produces a set of $2\kappa+3$ linear equations for the $2\kappa+3$ constants $a_j$, and we hope that
these satisfy~\eqref{eq:A17} for every $\gamma$. We choose $\kappa:=10$ and $\Upsilon:=\{0.5, 1.5, 2,
2.4, 2.8, 7.9, 18, 10^2, 10^3 ,10^5\}$. Finally, with an abuse of notation we take for $a_j$ the solution
of the system, rounded above to $10^{-7}$: this produces the numbers in Table~\ref{tab:A2}.
\begin{table}[H]
\caption{Values of the coefficients.}\label{tab:A2}
\smallskip
{
\begin{tabular}{|r|r||r|r|}
  \toprule
  $j$     & \mltc{1}{c||}{$a_j\cdot 10^7$}    &  $j$    &  \mltc{1}{c|}{$a_j\cdot 10^7$} \\
  \midrule
  $ 1$    & $             -10417203$          &  $13$   & $  -18920268046344982450$      \\
  $ 2$    & $            1056404889$          &  $14$   & $   29659178484686316889$      \\
  $ 3$    & $          -65191418930$          &  $15$   & $  -37103060687919097856$      \\
  $ 4$    & $         2306235683461$          &  $16$   & $   36963001195180424340$      \\
  $ 5$    & $       -50953892956052$          &  $17$   & $  -29124459758424138052$      \\
  $ 6$    & $       745294415104297$          &  $18$   & $   17917680016161661642$      \\
  $ 7$    & $     -7554469767270438$          &  $19$   & $   -8424311293805783518$      \\
  $ 8$    & $     55069155554895360$          &  $20$   & $    2923218093750242944$      \\
  $ 9$    & $   -297487524612176257$          &  $21$   & $    -705518033170496127$      \\
  $10$    & $   1219731091815491142$          &  $22$   & $     105765338120745449$      \\
  $11$    & $  -3866974934911032963$          &  $23$   & $      -7417073631321810$      \\
  $12$    & $   9612711864719121022$          &         &                                \\
  \bottomrule
\end{tabular}
}
\end{table}
Then, using Sturm's algorithm, we prove that the values found actually give an upper bound for $g$, so
that~\eqref{eq:A18} holds with such $a_j$'s. These constants verify
\begin{equation}\label{eq:A19}
\begin{aligned}
&\sum_j a_j                                                =     1.4999\ldots,\\
&\sum_j a_j\frac{\Gamma'}{\Gamma}\Big(\frac{s_j}{2}\Big)   \leq  0.6552,
\end{aligned}
\qquad\qquad
\begin{aligned}
&\sum_j a_j\Big(\frac{2}{s_j}+\frac{2}{s_j-1}\Big)         \leq -1.577,       \\
&\sum_j a_j\frac{\Gamma'}{\Gamma}\Big(\frac{s_j+1}{2}\Big) \leq  0.7314.
\end{aligned}
\end{equation}
We write $\sum_j a_j\frac{\zeta'_\K}{\zeta_\K}(s_j)$ as
\[
-\sum_n\tilde\Lambda_\K(n)S(n)\quad\text{with}\quad S(n):=\sum_j \frac{a_j}{n^{s_j}}.
\]
We check numerically that $S(n)<0$ for $n\leq 10284$ with the exception of $S(4)$, which is in any case
$\leq 0.0237$. Then, since the sign
of $a_j$ alternates, we can easily prove that each pair $\frac{a_1}{n^{s_1}} + \frac{a_2}{n^{s_2}}$,
\ldots, $\frac{a_{2q+1}}{n^{s_{2q+1}}} + \frac{a_{2q+2}}{n^{s_{2q+2}}}$ and the last term
$\frac{a_{2q+3}}{n^{s_{2q+3}}}$ are negative for every $n\geq 10284$, thus
\begin{align}
\sum_j a_j\frac{\zeta'_\K}{\zeta_\K}(s_j)
&= -\sum_n\tilde\Lambda_\K(n)S(n)
 \leq -n_\K\sum_{n\neq 4}\Lambda(n)S(n)                        \label{eq:A20}\\
&=    -n_\K\Big[\sum_{n=1}^\infty\Lambda(n)S(n)
                -\Lambda(4)S(4)
           \Big]                                                \notag\\
&=     n_\K\Big[\sum_{j}a_j\frac{\zeta'}{\zeta}(s_j)
                +\Lambda(4)S(4)
           \Big]
 \leq 1.3372 n_\K.                                              \notag
\end{align}
The result now follows from~\eqref{eq:A16}, and~(\ref{eq:A18}--\ref{eq:A20}).
\end{proof}

Thus, by~\eqref{eq:A15} and Lemmas~\ref{lem:A7} and~\ref{lem:A8} we get
\begin{equation}
\begin{aligned}
|M_{a,q}\Delta_{2,h}\Sigma_{\chi^*,1}|
\leq&
\frac{4x^{3/2}}{\eul(q)}\sum_{\substack{\tfrac{1}{2}+i\gamma\in Z_{q}\\ 5\leq |\gamma|\leq T}}
               \frac{\sin^2\big(\frac{\gamma h}{2x}\big)}{|(\tfrac{1}{2}+i\gamma)(\tfrac{3}{2}+i\gamma)|}       \\
   &+ \frac{h^2}{\sqrt{x}}\Big(\frac{1}{\pi}\log(q^2T)\log T + 3.43\log q  - 2.699 + \frac{20.1}{\eul(q)}\Big). \label{eq:A21}
\end{aligned}
\end{equation}
%
%
To bound the sum by partial summation we need a formula for $N_{q}(T)$, the number of zeros $\rho$ of
$\zeta_\K$ with $\Ree(\rho)\in(0,1)$ and $|\Imm(\rho)|\leq T$. Let $W_{q}(T):=
\frac{T}{\pi}\log\big((\frac{ T}{2\pi e})^{\eul(q)}\dq\big)$, and let $\Rp_{q}(T) := N_{q}(T) -
W_{q}(T)$. Then
\begin{align*}
\Big|\frac{1}{\eul(q)}N_{q}(T) - \frac{T}{\pi}\log\Big(\frac{T}{2\pi e}(\dq)^{1/\eul(q)}\Big)\Big|
&=
\frac{|\Rp_{q}(T)|}{\eul(q)}\\
&\leq d_1\log\Big(\frac{T}{2\pi}(\dq)^{1/\eul(q)}\Big)
    + d_2
    + \frac{d_3}{\eul(q)}
 =: \frac{R_{q}(T)}{\eul(q)}
\qquad T\geq 5
\end{align*}
with $d_1 = 0.395$, $d_2 = 3.459$ and $d_3 = 2.559$ (this particular set of values is computed using the
algorithm of Trudgian~\cite{TrudgianIII} with $\eta=0.36$, $p=-\eta$ and $r=\frac{1+\eta-p}{1/2+\eta}=2$.)
Thus, by partial summation we get
\begin{align*}
\sum_{\smash[b]{\substack{\tfrac{1}{2}+i\gamma\in Z_{q}\\ 5 <|\gamma|\leq T}}} \frac{\sin^2\big(\frac{\gamma h}{2x}\big)}{\gamma^2}
=& \int_{5^+}^{T^+} \frac{\sin^2\big(\frac{\gamma h}{2x}\big)}{\gamma^2} \d N_q(\gamma)                 \\
=& \frac{\sin^2\big(\frac{hT}{2x}\big)}{T^2} N_q(T)
  -\frac{\sin^2\big(\frac{5h}{2x}\big)}{5^2} N_q(5^+)
  -\int_{5}^{T} \Big[\frac{\sin^2\big(\frac{\gamma h}{2x}\big)}{\gamma^2}\Big]' N_q(\gamma)\d\gamma     \\
=& \frac{\sin^2\big(\frac{hT}{2x}\big)}{T^2} \Rp_q(T)
  -\frac{\sin^2\big(\frac{5h}{2x}\big)}{5^2} \Rp_q(5^+)                                                 \\
&+ \int_{5}^{T} \frac{\sin^2\big(\frac{\gamma h}{2x}\big)}{\gamma^2}\log\Big(\frac{\gamma^{\eul(q)}\dq}{2\pi}\Big)\frac{\d\gamma}{\pi}
 - \int_{5}^{T} \Big[\frac{\sin^2\big(\frac{\gamma h}{2x}\big)}{\gamma^2}\Big]' \Rp_q(\gamma)\d\gamma.
\end{align*}
Recalling the upper bound $|\Rp_q(T)|\leq R_q(T)$ we get
\begin{align*}
\frac{1}{\eul(q)}&\sum_{5<|\gamma|\leq T} \frac{\sin^2\big(\frac{\gamma h}{2x}\big)}{\gamma^2}
\leq \frac{R_q(T)}{\eul(q)T^2}
 +\Big(\frac{h}{2x}\Big)^2 \frac{R_q(5)}{\eul(q)}\\
&+\int_{5}^{T} \frac{\sin^2\big(\frac{\gamma h}{2x}\big)}{\gamma^2} \log\Big(\frac{\gamma(\dq)^{1/\eul(q)}}{2\pi}\Big)\frac{\d\gamma}{\pi}
 +\frac{1}{\eul(q)}\int_{5}^{T} \Big|\frac{h}{2x}\frac{\sin\big(\frac{\gamma h}{x}\big)}{\gamma^2} - 2\frac{\sin^2\big(\frac{\gamma h}{2x}\big)}{\gamma^3}\Big| R_q(\gamma)\d\gamma.
\end{align*}
Using the inequality $|\frac{\sin(2v)}{2}-\frac{\sin^2 v}{v}|\leq\frac{4}{5}$,
%
%
we simplify to get
\begin{align*}
\frac{1}{\eul(q)}\sum_{5<|\gamma|\leq T} \frac{\sin^2\big(\frac{\gamma h}{2x}\big)}{\gamma^2}
\leq& \frac{R_q(T)}{\eul(q)T^2}
 +\Big(\frac{ h}{2x}\Big)^2 \frac{R_q(5)}{\eul(q)}\\
&+\int_{5}^{T} \frac{\sin^2\big(\frac{\gamma h}{2x}\big)}{\gamma^2} \log\Big(\frac{\gamma(\dq)^{1/\eul(q)}}{2\pi}\Big)\frac{\d\gamma}{\pi}
 +\frac{4h}{5\eul(q)x} \int_{5}^{T} \frac{R_q(\gamma)}{\gamma^2}\d\gamma,
\intertext{and since $\int_{5}^{+\infty} \frac{R_q(\gamma)}{\gamma^2}\d\gamma\leq 0.079\log\dq +
0.7528\eul(q) + 0.5118$, we get}
%
%
%
\frac{1}{\eul(q)}\sum_{5<|\gamma|\leq T} \frac{\sin^2\big(\frac{\gamma h}{2x}\big)}{\gamma^2}
\leq&
 \frac{h}{2\pi x}\int_{0}^{\frac{hT}{2x}} \frac{\sin^2 t}{t^2} \d t \log\Big(\frac{T(\dq)^{1/\eul(q)}}{2\pi}\Big)\\
&    + \frac{(0.253\log\dq + 2.409\eul(q) + 1.638)h}{4\eul(q)x}
     + \frac{R_q(T)}{\eul(q)T^2}
     + \Big(\frac{ h}{2x}\Big)^2 \frac{R_q(5)}{\eul(q)}.
%
%
\end{align*}
By Lemma~\ref{lem:A5} and the bound $R_q(T)\leq 0.395\log(T^{\eul(q)}\dq) + 2.74\eul(q) + 2.559$ for every
$T\geq 5$,
%
%
the bound becomes
\begin{align*}
\frac{1}{\eul(q)}&\sum_{5<|\gamma|\leq T}\frac{\sin^2\big(\frac{\gamma h}{2x}\big)}{\gamma^2}\\
%
\leq&  \frac{h}{4 x}\Big(1 - \frac{2}{\pi}\frac{x}{hT} + \frac{2}{\pi}\frac{x^2}{h^2T^2}\Big)\log\Big(\frac{T(\dq)^{1/\eul(q)}}{2\pi}\Big)
     + \frac{(0.253\log\dq + 2.409\eul(q) + 1.638)h}{4\eul(q)x}\\
&    + \frac{0.395\log(T^{\eul(q)}\dq) + 2.74\eul(q) + 2.559}{\eul(q)T^2}
     + \Big(\frac{ h}{2x}\Big)^2 \frac{0.395\log(5^{\eul(q)}\dq) + 2.74\eul(q) + 2.559}{\eul(q)}.
\intertext{We substitute $\log\dq=\eul(q)\log q -\eul(q)\sum_{p|q}\frac{\log p}{p-1}$ in the first two
terms, while for the last two we simply use the bound $\log\dq \leq \eul(q)\log q$. Moreover, since for
$q\geq 3$ we have $\eul(q)\geq 2$, so we use this hypothesis to simplify the terms decaying as $1/T^2$.
We get}
\leq&  \frac{h}{4 x}\Big(1 - \frac{2}{\pi}\frac{x}{hT} + \frac{2}{\pi}\frac{x^2}{h^2T^2}\Big)\log\Big(\frac{qT}{2\pi}\Big)
     + \Big(0.253\log q - \sum_{p|q}\frac{\log p}{p-1} + 2.409 + \frac{1.638}{\eul(q)}\Big)\frac{h}{4x}\\
&    + \frac{0.395\log(qT) + 4.02}{T^2}
     + \Big(0.395\log q + 3.376 + \frac{2.559}{\eul(q)}\Big)\frac{h^2}{4x^2},
\end{align*}
where we used that $1.253-\frac{2}{\pi}Y+\frac{2}{\pi}Y^2 \geq 1$ to simplify the coefficient of $\sum_{p|q}$.
%
%
%
Thus~\eqref{eq:A21} becomes
\begin{equation}
\begin{aligned}
|M_{a,q}\Delta_{2,h}\Sigma_{\chi^*,1}|
\leq
&h\sqrt{x}\log\Big(\frac{qT}{2\pi}\Big)
 -\frac{2}{\pi}h\sqrt{x}\frac{x}{hT}\log\Big(\frac{qT}{2\pi}\Big)
 +\frac{2}{\pi}h\sqrt{x}\frac{x^2}{h^2T^2}\log\Big(\frac{qT}{2\pi}\Big)           \\
&+ \Big(0.253\log q - \sum_{p|q}\frac{\log p}{p-1} + 2.409 + \frac{1.638}{\eul(q)}\Big)h\sqrt{x}
 + (1.58\log(qT) + 16.08)\frac{x^{3/2}}{T^2}                                      \\
&+ \Big(\frac{1}{\pi}\log(q^2T)\log T + 3.9\log q + 0.7 + \frac{22.7}{\eul(q)}\Big)\frac{h^2}{\sqrt{x}}.\label{eq:A22}
\end{aligned}
\end{equation}
%
%

\section{Proof of Theorem~\ref{th:A1}}
\label{sec:A5}
%
%
Substituting~\eqref{eq:A22} into~\eqref{eq:A14} and by~\eqref{eq:A13} we get
\begin{align*}
\sum_{\smash[b]{\substack{|p-x|<h\\ p=a\Pmod{q}}}} \log p
\geq& \frac{h}{\eul(q)}
 - \Big[\sqrt{x}\log\Big(\frac{qT}{2\pi}\Big)
        - \frac{2}{\pi}\sqrt{x}\frac{x}{hT}\log\Big(\frac{qT}{2\pi}\Big)
        + \frac{2}{\pi}\sqrt{x}\frac{x^2}{h^2T^2}\log\Big(\frac{qT}{2\pi}\Big)                \\
&       + \Big(0.253\log q - \sum_{p|q}\frac{\log p}{p-1} + 2.409 + \frac{1.638}{\eul(q)}\Big)\sqrt{x}
        + (1.58\log(qT) + 16.08)\frac{x^{3/2}}{hT^2}                                          \\
&       + \Big(\frac{1}{\pi}\log(q^2T)\log T + 3.9\log q + 0.7 + \frac{22.7}{\eul(q)}\Big)\frac{h}{\sqrt{x}}
   \Big]                                                                                      \\
&- 4\Big(x^{3/2}+\frac{h^2}{4\sqrt{x}}\Big)\Big(1 + \frac{2.89}{T}\Big)\frac{\log(qT)}{\pi hT}\\
&- 0.95\sqrt{x}
 - 3.7\sqrt[3]{x}
 - \omega(q)\log(2x)
 - 1.7\frac{h}{x}.
\end{align*}
We introduce a new parameter $\beta$ defined as $hT=:\beta x$. Thus, the previous inequality becomes
\begin{equation}
\begin{aligned}
&\frac{1}{\sqrt{x}}\sum_{\smash[b]{\substack{|p-x|<h\\ p=a\Pmod{q}}}} \log p
\geq \frac{h/\sqrt{x}}{\eul(q)}
 - \Big(1 + \frac{2}{\pi\beta} + \frac{2}{\pi\beta^2} + \frac{4\cdot 2.89}{\pi\beta T}\Big)\log(qT)
 - 0.253\log q                                                                                     \\
&- \Big(\frac{1}{\pi}\log(q^2T)\log T
        + 3.9\log q
        + 0.7
        + \frac{22.7}{\eul(q)}
        + \frac{1.58\log(qT) + 16.08}{\beta^2}
        + \Big(1 + \frac{2.89}{T}\Big)\frac{\log(qT)}{\pi T}
   \Big)\frac{h}{x}                                                                                \\
&- 1.53
 - \sum_{p|q}\frac{\log p}{p-1}
 - \frac{1.638}{\eul(q)}
 - \frac{2\log(2\pi)}{\pi\beta}\Big(1-\frac{1}{\beta}\Big)
 - 3.7 x^{-1/6}
 - \omega(q)\frac{\log(2x)}{\sqrt{x}}
 - 1.7\frac{h}{x^{3/2}}.\label{eq:A23}
\end{aligned}
\end{equation}
%
We simplify this formula by noticing that for $\beta\geq 20$ and $x\geq (10\eul(q)\log q)^2$
(unfortunately we cannot hope to prove anything as strong as this one, so that these assumptions will be
satisfied), the function appearing in the last line is larger than $-2$ for $q\geq 18$ (we use the
assumption $\frac{h}{x}\leq \frac{5}{6}$ to bound $\frac{h}{x^{3/2}}$ with $\frac{5/6}{\sqrt{x}}$, and
when $q\geq 800$ we apply the bounds $\omega(q)\leq \log q$ and $\eul(q)\geq \sqrt{q}$).
%
%
Thus we have
\begin{equation}\label{eq:A24}
\begin{aligned}
&\frac{1}{\sqrt{x}}\sum_{\smash[b]{\substack{|p-x|<h\\ p=a\Pmod{q}}}} \log p
\geq \frac{h/\sqrt{x}}{\eul(q)}
 - \Big(1 + \frac{2}{\pi\beta} + \frac{2}{\pi\beta^2} + \frac{4\cdot 2.89}{\pi\beta T}\Big)\log(qT)
 - 0.253\log q
 - 2                                                                                               \\
&- \Big(\frac{1}{\pi}\log(q^2T)\log T
        + 3.9\log q
        + 0.7
        + \frac{22.7}{\eul(q)}
        + \frac{1.58\log(qT) + 16.08}{\beta^2}
        + \Big(1 + \frac{2.89}{T}\Big)\frac{\log(qT)}{\pi T}
   \Big)\frac{h}{x}.
\end{aligned}
\end{equation}
We introduce three nonnegative parameters $\alpha$, $\delta$ and $\rho$, and we further set
\[
h = \eul(q)(\alpha\log x + \delta\log q + \rho)\sqrt{x},
\qquad
T = \frac{\beta}{\eul(q)}\frac{\sqrt{x}}{\alpha\log x + \delta\log q + \rho}.
\]
For the first part of the theorem, that is, the existence of a prime $p=a\pmod{q}$ with $|p-x|\leq h$, it is
sufficient to prove that the function appearing on the right hand side of~\eqref{eq:A24} is positive.
This happens when
\begin{equation}\label{eq:A25}
(1 - F)(\alpha\log x + \delta\log q + \rho) > G
\end{equation}
where
\begin{align*}
F(q,x)&:= \!\Big(\frac{\log(q^2T)\log T}{\pi}
           + 3.9\log q
           + 0.7
           + \frac{22.7}{\eul(q)}
           + \frac{1.58\log(qT) + 16.08}{\beta^2}
           + \Big(1 + \frac{2.89}{T}\Big)\frac{\log(qT)}{\pi T}
          \!\Big)\frac{\eul(q)}{\sqrt{x}},                       \\
G(q,x)&:= \!\Big(1 + \frac{2}{\pi\beta} + \frac{2}{\pi\beta^2} + \frac{4\cdot 2.89}{\pi\beta T}\Big)\log(qT)
           + 0.253\log q + 2.
\end{align*}
We still have to make a choice for $\beta$, for which we have two different requirements.\\
{\sf CASE 1}. Consider $x\to\infty$, for a fixed $q$. Then $\log T \sim \frac{1}{2}\log x$, as soon as
$\log\beta = o(\log x)$. Thus $F \ll \frac{\log^2 x}{\sqrt{x}}$, and to prove~\eqref{eq:A25} we need
\[
\alpha\log x + \rho
>
\Big(\frac{1}{2} + \frac{1}{\pi\beta} + \frac{1}{\pi\beta^2}\Big)\log x + O(1),
\]
not uniformly in $q$ and in the other parameters. Thus we need
\[
\alpha > \frac{1}{2} + \frac{1}{\pi\beta} + \frac{1}{\pi\beta^2},
\]
and we can improve this bound to $\alpha \geq \frac{1}{2}$ if we assume that $\beta \asymp \log x$, at
the cost of increasing $\rho$.\\
{\sf CASE 2}. Consider $q\to\infty$, and $x= x_0(q) = (m\eul(q)\log q)^2$ for some constant $m$. Then
\[
T = \frac{m\beta}{2\alpha+\delta} + O\Big(\frac{\log\log q}{\log q}\Big),
\]
not uniformly in $\alpha$, $\delta$, $\rho$ and $m$. In particular, it stays bounded if we assume that
$\beta$ is bounded, and
\[
F = \frac{1}{m}
     \Big(\frac{2}{\pi}\log T
        + 3.9
        + \frac{1.58}{\beta^2}
        + \Big(1 + \frac{2.89}{T}\Big)\frac{1}{\pi T}
        + O\Big(\frac{1}{\log q}\Big)
     \Big).
\]
Thus $F$ is small if $m$ is large enough, and~\eqref{eq:A25} is implied by
\[
(2\alpha+\delta)\log q  + \rho
>
(1 - F)^{-1}\Big(1.253 + \frac{2}{\pi\beta} + \frac{2}{\pi\beta^2} + \frac{4\cdot 2.89}{\pi\beta T}\Big)\log q + O(1),
\]
because $\eul(q)\log q\gg q$.
%
%
Thus it is sufficient to have
\[
2\alpha+\delta
\geq
\Big(1
     - \frac{1}{m}\Big(\frac{2}{\pi}\log T
                       + 3.9
                       + \frac{1.58}{\beta^2}
                       + \Big(1 + \frac{2.89}{T}\Big)\frac{1}{\pi T}
                  \Big)
\Big)^{-1}\Big(1.253 + \frac{2}{\pi\beta} + \frac{2}{\pi\beta^2} + \frac{4\cdot 2.89}{\pi\beta T}\Big),
\]
at the cost of increasing $\rho$.

In order to meet both requirements for $\beta$ we set
\begin{equation}\label{eq:A26}
\beta = \ell \log\Big(\frac{\sqrt{x}}{\eul(q)\log q}\Big),
\end{equation}
for a suitable constant $\ell>0$ that we will fix later. In this way we can set $\alpha=1/2$, and
$\delta$ will be close to $0.253$, specifically: $|\delta-0.253|\ll \frac{\log(\ell m)}{m} +
\frac{1}{\ell\log m}$. Obviously we are interested in producing small values for $m$. Thus, for a fixed
value of $\alpha$ and $\delta$ we select the value of $\ell$ producing the minimum $m$ such that
$(\delta,\ell,m)$ satisfies the requirements.\\
If one is interested mainly in the $q$ aspect, then one can select $\alpha = 1.253/2$; in this way
$\delta$ can be chosen arbitrarily small if $m$ and $\ell$ are large enough, and the value $\delta=0$ is
possible for every $\alpha > 1.253/2$. Possible choices are in Table~\ref{tab:A3}.

The previous argument has showed how we have to set $\beta$, and what we can expect to be able to prove.
However, in order to get a true proof we need to convert~\eqref{eq:A25} into something decreasing in $x$
when all other parameters are fixed, because only in this way can we prove the claim for all $x\geq x_0$
by testing it only in $x_0$.\\
We notice that according to our definitions both $\beta$ and $T$ increase as functions of $x$, at least
for $x\geq 10$. Moreover, setting $u:=\sqrt{x}$, one sees that $\frac{1}{u}\log^2T$ decreases if and only
if
\begin{gather*}
\frac{2\log T}{u}\frac{d\log T}{du} \leq \frac{\log^2T}{u^2}
\iff
2u\frac{d\log T}{du} \leq \log T
\iff
2[\frac{1}{\log u} + 1 - \frac{2\alpha}{(2\alpha \log u + \delta\log q+\rho)}] \leq \log T\\
\Leftarrow
2 + \frac{4}{\log x} \leq \log T.
\end{gather*}
For $x\geq 100$, this is true whenever $T\geq 20$. This suffices to prove that in this range $F(q,x)$
decreases as a function of $x$. Unfortunately this is false for $G$, thus we have to modify it into a new
$\overline{G}$ having a better behavior in $x$ and such that $\overline{G}\geq G$ so that
\begin{equation}\label{eq:A27}
(1 - F)(\alpha\log x + \delta\log q + \rho) > \overline{G}
\end{equation}
implies~\eqref{eq:A25}.\\
Firstly, we notice that for $x=:u^2$ moderately larger than $100$, the function $\frac{1}{u}\log^2T\log
u$ decreases as well. In fact, this happens if and only if
\begin{multline}\label{eq:A28}
\frac{2\log T\log u}{u}\frac{d\log T}{du} + \frac{\log^2T}{u^2} \leq \frac{\log^2T\log u}{u^2}
\iff
2u\frac{d\log T}{du} + \frac{\log T}{\log u} \leq \log T\\
\iff
2\Big[\frac{1}{\log u} + 1 - \frac{2\alpha}{(2\alpha \log u + \delta\log q+\rho)}\Big] \leq \Big[1 - \frac{1}{\log u}\Big]\log T
\Leftarrow
2\frac{\log x + 2}{\log x - 2} \leq \log T
\end{multline}
and for $x\geq 23000$ this is true whenever $T\geq 20$, once again.
%
%
This proves that in this range also $F(q,x)\log x$ decreases as a function of $x$. Secondly, recalling
our setting for $T$ and $\beta$, we see that
\begin{align*}
qT
=    \frac{q\ell\sqrt{x}}{\eul(q)}\frac{\log\big(\frac{\sqrt{x}}{\eul(q)\log q}\big)}{\alpha\log x + \delta \log q + \rho}
\leq \frac{q\ell\sqrt{x}}{2\alpha\eul(q)}.
\end{align*}
We use this bound to substitute $\log(qT)$ in $G$, producing
\[
\overline{G}(q,x)
:= \Big(1 + \frac{2}{\pi\beta} + \frac{2}{\pi\beta^2} + \frac{4\cdot 2.89}{\pi\beta T}\Big)\log\Big(\frac{q\ell\sqrt{x}}{2\alpha\eul(q)}\Big)
   + 0.253\log q + 2.
\]
With this $\overline{G}$, Inequality~\eqref{eq:A27} may be written as
\begin{multline}\label{eq:A29}
\Big(\alpha-\frac{1}{2}\Big)\log x + (1 - F)(\delta\log q + \rho)
>  \alpha F\log x
  + \Big(\frac{1}{\pi\beta} + \frac{1}{\pi\beta^2} + \frac{2\cdot 2.89}{\pi\beta T}\Big)\log x \\
  + \Big(1 + \frac{2}{\pi\beta} + \frac{2}{\pi\beta^2} + \frac{4\cdot 2.89}{\pi\beta T}\Big)\log\Big(\frac{q\ell}{2\alpha\eul(q)}\Big)
  + 0.253\log q + 2.
\end{multline}
When $\alpha\geq 1/2$, the function appearing on the left hand side increases in $x$ (whenever $x\geq
100$, $T\geq 20$), while the function on the right hand side decreases in $x$ (whenever $x\geq 23000$,
$T\geq 20$). This shows that if $x\geq 23000$ and $\alpha\geq 1/2$, we can check~\eqref{eq:A29} (and
hence~\eqref{eq:A27}, since they are equivalent) for $x\geq x_0$ by testing it for $x_0$.\\
We also have to satisfy the assumption
\begin{equation}\label{eq:A30}
\frac{1}{T}
   = \frac{\eul(q)}{\sqrt{x}}\frac{\alpha\log x + \delta\log q + \rho}{\ell\log(\frac{\sqrt{x}}{\eul(q)\log q})}
\leq \frac{1}{20},
\end{equation}
and, since we have assumed $h\leq 5x/6$ in several places, we need also
\begin{equation}\label{eq:A31}
\frac{h}{x} = \frac{\eul(q)}{\sqrt{x}}(\alpha\log x + \delta\log q + \rho)
           \leq \frac{5}{6},
\end{equation}
where again the functions appearing on the left hand sides decrease in $x$ (for $x\geq e^2$).

The combinations of values for the parameters $\alpha$, $\delta$, and $m$ in Table~\ref{tab:A3} are in
some sense unrealistic: they can be satisfied only for extremely large $q$. In order to have a claim
which could be proved for every $q$ we have to increase $m$ and choose $\rho$ accordingly. Our choices
are in Table~\ref{tab:A4}, and for every choice of the parameters appearing there we verify by direct
computation that all requirements are satisfied by $x = x_0(q) := (m\eul(q)\log q)^2$ when $3\leq q \leq
q_0$, with just a few exceptions which are in Table~\ref{tab:A5} and for which we have to test the claim
directly for $x\in [x_0(q),x(q)]$.

To deal with larger $q$'s, we set $x=x_0(q)=(m\eul(q)\log q)^2$ in~\eqref{eq:A27}, but, again, we have to
modify $F_0(q)=F(q,x_0(q))$ and $\overline{G}_0(q)=\overline{G}(q,x_0(q))$ in order to produce an
inequality which will hold for every $q\geq q_0$ when verified for $q_0$. For this purpose we introduce
\begin{align*}
\tilde{F}_0(q)
&:= \frac{1}{m}
    \Big(\frac{\log(q^2T_+)\log T_+}{\pi\log q}
        + 3.9
        + \frac{0.7}{\log q}
        + \frac{22.7}{q}
        + \frac{1.58\log(qT_+) + 16.08}{\beta_0^2\log q}
        + \Big(1 + \frac{2.89}{T_-}\Big)\frac{\log(qT_+)}{\pi T_-\log q}
   \Big)
\intertext{and}
\tilde{G}_0(q)
&:= \Big(1
        + \frac{2}{\pi\beta_0}
        + \frac{2}{\pi\beta_0^2}
        + \frac{4\cdot 2.89}{\pi\beta_0 T_-}
   \Big)\log\Big(\frac{\ell m q\log q}{2\alpha}\Big)
   + 0.253\log q + 2
\end{align*}
with
\[
\beta_0 := \ell\log m,
\qquad
T_- := \frac{\beta_0 m\log q}{2\alpha\log(m q\log q)+\delta\log q + \rho},
\qquad
T_+ := \frac{\beta_0 m}{2\alpha+\delta}.
\]
Then for $x=x_0(q)$ one has $\beta=\beta_0$, $T_-\leq T\leq T_+$, $\tilde{F}_0(q)\geq F_0(q)$ and
$\tilde{G}_0(q)\geq \overline{G}_0(q)$, so that~\eqref{eq:A27} for $x=x_0(q)$ holds for sure if
\begin{equation}\label{eq:A32}
(1 - \tilde{F}_0(q))(2\alpha\log(m q) + \delta\log q + \rho) \geq \tilde{G}_0(q).
\end{equation}
We notice that $\tilde{F}_0(q)$ and $1/T_-$ decrease in $q$, thus~\eqref{eq:A32} may be written as
\begin{multline*}
\Big((1 - \tilde{F}_0(q))(2\alpha + \delta)
     - \Big(1.253
            + \frac{2}{\pi\beta_0}
            + \frac{2}{\pi\beta_0^2}
            + \frac{4\cdot 2.89}{\pi\beta_0 T_-}
       \Big)
\Big)\log q
\geq
   \Big(1
        + \frac{2}{\pi\beta_0}
        + \frac{2}{\pi\beta_0^2}
        + \frac{4\cdot 2.89}{\pi\beta_0 T_-}
   \Big)\log\log q                           \\
 + (\tilde{F}_0(q)-1)(2\alpha\log m + \rho)
 + \Big(1
        + \frac{2}{\pi\beta_0}
        + \frac{2}{\pi\beta_0^2}
        + \frac{4\cdot 2.89}{\pi\beta_0 T_-}
   \Big)\log\Big(\frac{\ell m}{2\alpha}\Big)
   + 2
\end{multline*}
i.e., as
\begin{equation}\label{eq:A33}
A \log q - B \log\log q - C \geq 0
\end{equation}
where $A$ increases in $q$ and $B$ and $C$ decrease. The function on the left hand side is increasing in
$q$ when
\[
A' \log q - B' \log\log q + \frac{A}{q} - \frac{B}{q\log q} - C' > 0,
\]
and for this it is sufficient to have
\[
-\tilde{F}'_0(q)(2\alpha+\delta) \log q + \frac{A}{q} - \frac{B}{q\log q} > 0.
\]
Since
\begin{equation}\label{eq:A34}
-\tilde{F}'_0(q)
\geq \frac{S}{q\log^2 q}
\quad\text{with}\quad
S:=\frac{1}{m}\Big(\frac{\log^2 T_+}{\pi}
                   + 0.7
                   + \frac{1.58\log T_+ + 16.08}{\beta_0^2}
              \Big),
\end{equation}
in order to have a monotonous behavior of~\eqref{eq:A33} it is sufficient to have
\begin{equation}\label{eq:A35}
A\log q \geq B - S(2\alpha+\delta).
\end{equation}
In this way we see that if~\eqref{eq:A33} holds for a certain $q=q_0$ large enough to
satisfy~\eqref{eq:A35}, then it is proved for every $q\geq q_0$. Moreover, we notice that
inequalities~\eqref{eq:A30} and~\eqref{eq:A31} in $x_0(q)= (m\eul(q)\log q)^2$ are satisfied as soon as
\begin{equation}\label{eq:A36}
\frac{1}{T_-}
=    \frac{2\alpha\log(m q\log q) + \delta\log q + \rho}{\ell m\log m\log q}
\leq \frac{1}{20},
\end{equation}
and
\begin{equation}\label{eq:A37}
\frac{2\alpha\log(m q\log q) + \delta\log q + \rho}{m\log q}
\leq \frac{5}{6}.
\end{equation}
Thus (finally!) we have produced the test we were looking for: we search for a $q_0$
satisfying~\eqref{eq:A32}, \eqref{eq:A35}, \eqref{eq:A36} and~\eqref{eq:A37}. Then everything is proved
for $q\geq q_0$. Our computations show that the values of $q_0$ appearing in Table~\ref{tab:A4} pass this
test.

For $q\leq q_0$ and $x\in[x_0(q),x(q)]$ we use the mighty computer procedure~\texttt{Check1} described
below so that now the proof of the first claim of the theorem is complete.

\bigskip
%
For the second part of the theorem, i.e. the claim ensuring that if we increase $\alpha$ by one then
there are at least $\sqrt{x}$ primes $p=a\pmod{q}$ in $|p-x|\leq h$, we proceed in similar way. Indeed,
the inequality
\[
\log(x+h)\sum_{\smash[b]{\substack{|p-x|<h\\ p=a\Pmod{q}}}} 1
\geq \sum_{\smash[b]{\substack{|p-x|<h\\ p=a\Pmod{q}}}} \log p,
\]
allows to prove the claim by proving that the function appearing on the right hand side
of~\eqref{eq:A23} is larger than $\log(x+h)$. This amounts to modifying~\eqref{eq:A25} into
\[
(1 - F)((\alpha+1)\log x + \delta\log q + \rho) > G + \log(x+h),
\]
i.e. into
\[
(1 - F)(\alpha\log x + \delta\log q + \rho) > G + F\log x + \log(1+h/x),
\]
where $F$ and $G$ are defined as before (but with $\alpha+1$ instead of $\alpha$ in the definition of
$T$). We simplify the inequality recalling that we are assuming that $h/x \leq 5/6$. Moreover, we once
again use $\overline{G}$ instead of $G$ in order to get an inequality which is proved for all $x$ larger
than $x_0$ when it is proved for $x_0$: by~\eqref{eq:A28} this happens at least whenever $x\geq 23000$.
Thus it is sufficient to prove that
\begin{equation}\label{eq:A38}
(1 - F)(\alpha\log x + \delta\log q + \rho) > \overline{G} + F\log x + \log(11/6).
\end{equation}
Setting $x = x'_0(q)=(m'\eul(q)\log q)^2$, for a diverging $q$ the inequality becomes
\[
2(\alpha - (\alpha +1)F)\log(m'\eul(q)\log q)  + (1 - F)\delta\log q + O(1)
> \Big(1.253 + \frac{2}{\pi\beta} + \frac{2}{\pi\beta^2} + \frac{4\cdot 2.89}{\pi\beta T}\Big)\log q + O(1).
\]
If we assume that $F\leq \alpha/(\alpha+1)$, then the lower bound $\eul(q)\log q \geq q$ shows that this is
\[
(2\alpha + \delta - (2\alpha + \delta + 2)F)\log q + O(1)
> \Big(1.253 + \frac{2}{\pi\beta} + \frac{2}{\pi\beta^2} + \frac{4\cdot 2.89}{\pi\beta T}\Big)\log q + O(1),
\]
which forces us to select $\alpha$, $\delta$, $l$ and $m'$ in such a way that
\[
2\alpha + \delta - (2\alpha + \delta + 2)F
\geq 1.253 + \frac{2}{\pi\beta_0} + \frac{2}{\pi\beta_0^2} + \frac{4\cdot 2.89}{\pi\beta_0 T}
\]
with $\beta_0 = \ell'\log m'$, $T = \frac{\beta m'}{2\alpha+\delta}$ and
\[
F = \frac{1}{m'}\Big(\frac{2}{\pi}\log T
                 + 3.9
                 + \frac{1.58}{\beta_0^2}
                 + \Big(1 + \frac{2.89}{T}\Big)\frac{1}{\pi T}
            \Big).
\]
This implies that for the combinations of $\alpha$ and $\delta$ we have already considered before we have
to select for $\ell'$ and $m'$ the values in Table~\ref{tab:A3}.
As before, in order to get a statement provable for all $q$ we have to further increase $m'$, for which
we select the values in Table~\ref{tab:A4}. Now, for every choice of the parameters in Table~\ref{tab:A4}
we verify by direct computation that~\eqref{eq:A30}, \eqref{eq:A31} (substituting $\alpha$, $m$ and
$\ell$ with $\alpha+1$, $m'$ and $\ell'$) and~\eqref{eq:A38} are satisfied by $x = x'_0(q) :=
(m'\eul(q)\log q)^2$ when $3\leq q \leq q'_0$, with just a few exceptions which are in Table~\ref{tab:A6}
and for which we test the claim directly for $x\in [x'_0(q),x'(q)]$. This proves this part of the theorem
for $q\leq q'_0$.

To deal with larger $q$'s, we set $x=x'_0(q)=(m'\eul(q)\log q)^2$ in~\eqref{eq:A38}, but, again, we
substitute $F$ and $\overline{G}$ with $\tilde{F}_0(q)$ and $\tilde{G}_0(q)$, getting
\[
(1 - \tilde{F}_0(q))(2\alpha\log(m'\eul(q)\log q) + \delta\log q + \rho)
> \tilde{G}_0(q) + 2\tilde{F}_0(q)\log(m'\eul(q)\log q) + \log(11/6).
\]
Assuming
\begin{equation}\label{eq:A39}
\tilde{F}_0(q) \leq \alpha/(\alpha+1),
\end{equation}
the inequality is implied by
\begin{equation}\label{eq:A40}
(1 - \tilde{F}_0(q))(2\alpha\log(m'q) + \delta\log q + \rho) \geq \tilde{G}_0(q) + 2\tilde{F}_0(q)\log(m'q) + \log(11/6),
\end{equation}
which is what we get substituting $\eul(q)\log q$ with its upper bound $q$. We write this inequality as
\begin{align*}
\Big(2\alpha + \delta - (2\alpha + \delta + 2)\tilde{F}_0(q)
    -& \Big(1.253
            + \frac{2}{\pi\beta_0}
            + \frac{2}{\pi\beta_0^2}
            + \frac{4\cdot 2.89}{\pi\beta_0 T_-}
       \Big)
\Big)\log q                                  \\
\geq&
   \Big(1
        + \frac{2}{\pi\beta_0}
        + \frac{2}{\pi\beta_0^2}
        + \frac{4\cdot 2.89}{\pi\beta_0 T_-}
   \Big)\log\log q
 + (\tilde{F}_0(q)-1)(2\alpha\log m' + \rho) \\
&+ \Big(1
        + \frac{2}{\pi\beta_0}
        + \frac{2}{\pi\beta_0^2}
        + \frac{4\cdot 2.89}{\pi\beta_0 T_-}
   \Big)\log\Big(\frac{\ell m'}{2\alpha}\Big)
   + 2\tilde{F}_0(q)\log m'
   + 2
   + \log(11/6),
\end{align*}
i.e. as
\begin{equation}\label{eq:A41}
{\mathcal A} \log q - {\mathcal B} \log\log q - {\mathcal C} \geq 0
\end{equation}
where ${\mathcal A}$ increases in $q$ and ${\mathcal B}$ and ${\mathcal C}$ decrease. It is monotonous in
$q$ when
\[
{\mathcal A}' \log q - {\mathcal B}' \log\log q + \frac{{\mathcal A}}{q} - \frac{{\mathcal B}}{q\log q} - {\mathcal C}' > 0,
\]
and for this it is sufficient to have
\[
-\tilde{F}'_0(q)(2\alpha+\delta+2) \log q + \frac{{\mathcal A}}{q} - \frac{{\mathcal B}}{q\log q} > 0.
\]
By~\eqref{eq:A34}, in order to have a monotonous behavior of~\eqref{eq:A41} it is sufficient to have
\begin{equation}\label{eq:A42}
{\mathcal A}\log q \geq {\mathcal B} - S(2\alpha+\delta+2).
\end{equation}
In this way we see that if~\eqref{eq:A41} holds for a certain $q=q'_0$ large enough to
satisfy~\eqref{eq:A42}, then it is proved for every $q\geq q'_0$. Thus we have produced the test we were
looking for: we search for the $q'_0$ satisfying~\eqref{eq:A36}, \eqref{eq:A37} (substituting $\alpha$,
$m$ and $\ell$ with $\alpha+1$, $m'$ and $\ell'$) \eqref{eq:A39}, \eqref{eq:A40} and \eqref{eq:A42}. Then
everything is proved for $q\geq q'_0$. Our computations show that each $q'_0$ appearing in
Table~\ref{tab:A4} pass this test, so that also the proof of the second claim of the theorem is
completed.

For $q\leq q'_0$ and $x\in[x'_0(q),x'(q)]$ we use the mighty computer procedure~\texttt{CheckSqrt}
described below so that now the proof of the theorem is complete.
\begin{remark*}
The procedures~\texttt{Check1} and~\texttt{CheckSqrt} check more than what is needed: they detect
the existence of prime numbers in $[x-h,x]$ except for the initial $x$'s.
\end{remark*}

\section{Proof of Theorems~\ref{th:A2} and~\ref{th:A3}}
\label{sec:A6}
%
We keep the notations
\[
h = \eul(q)(\alpha\log x + \delta\log q + \rho)\sqrt{x},
\qquad
T = \frac{\beta}{\eul(q)}\frac{\sqrt{x}}{\alpha\log x + \delta\log q + \rho},
\]
but we make a different choice for $\beta$. In fact, the first two negative terms
$(1+\frac{2}{\pi\beta})\log(qT)$ in~\eqref{eq:A23}, up to terms of lower order in $\beta$, are
\[
\log \beta + \frac{\log(q^2 x)}{\pi\beta}.
\]
This expression reaches its minimum when
\[
\beta = \frac{1}{\pi} \log(q^2 x),
\]
which is how we set $\beta$ now. This choice puts restrictions on $\alpha$ and $\delta$: to control the
terms appearing in the equations below we need to have $\alpha \geq 1/2$, $\delta>0$ and $2\alpha +\delta
\geq 2$. Since we are interested in furnishing small values for $\alpha$  and $\delta$, this leaves us
with the range $\alpha \in [1/2,1)$ and $\delta = 2-2\alpha$. In this range we pick the case
$\alpha=1/2$, $\delta=1$, which is a natural choice; the interested reader will be able to complete the
similar computations needed for any other setting of $\alpha$ and $\delta$. Thus, our settings are:
\[
h = \eul(q)(\tfrac{1}{2}\log(q^2 x) + \rho)\sqrt{x},
\qquad
T = \frac{\sqrt{x}}{\pi\eul(q)}\frac{\log(q^2 x)}{\frac{1}{2}\log(q^2 x) + \rho}.
\]
As a consequence we have
\begin{align}
T &\leq \frac{2}{\pi}\frac{\sqrt{x}}{\eul(q)},                            \label{eq:A43}\\
\log(qT) &\leq \frac{1}{2}\log(q^2x) + \log\Big(\frac{2}{\pi\eul(q)}\Big),\label{eq:A44}\\
\frac{2\log(qT)}{\pi\beta} &\leq 1.                                       \label{eq:A45}
\end{align}
Moreover,
\begin{align}
- \sum_{p|q}\frac{\log p}{p-1} - \log(\eul(q))
= - \log q - \sum_{p|q}\Big[\frac{\log p}{p-1} + \log\Big(1-\frac{1}{p}\Big)\Big]
\leq -\log q.                                                             \label{eq:A46}
\end{align}
%
%
The function appearing on the right hand side of~\eqref{eq:A23} is surely positive when
\begin{align*}
  \frac{1}{2}\log(q^2x) + \rho
\geq&
   \log(qT)
 + \frac{2}{\pi\beta}\log(qT)
 + 0.253\log q
 + 1.53
 - \sum_{p|q} \frac{\log p}{p-1}
 + \frac{1.638}{\eul(q)}                                                                         \\
&+ \frac{2\log(2\pi)}{\pi\beta}\Big(1 - \frac{1}{\beta}\Big)
 + \frac{2}{\pi\beta^2}\log(qT)
 + \frac{4\cdot 2.89}{\pi\beta}\frac{\log(qT)}{T}                                                \\
&+ \Big(\frac{1}{\pi}\log(q^2T)\log T + 3.9\log q  + 0.7 + \frac{22.7}{\eul(q)}\Big)\frac{h}{x}
 + (1.58\log(qT) + 16.08)\frac{h}{\beta^2x}                                                      \\
&+ 3.7 x^{-1/6}
 + \omega(q)\frac{\log(2x)}{\sqrt{x}}
 + \frac{h}{\pi x}\Big(1 + \frac{2.89}{T}\Big)\frac{\log(qT)}{T}
 + 1.7\frac{h}{x^{3/2}}.
\end{align*}
Using~\eqref{eq:A44} for the first $\log(qT)$, \eqref{eq:A45} for the terms $\frac{1}{\beta}\log(qT)$,
\eqref{eq:A43} for $\log(q^2T)\log T$, and~\eqref{eq:A46}, we deduce that it is sufficient to have

\begin{align*}
\rho
\geq&
   \log\Big(\frac{2}{\pi}\Big)
 + 1
 - 0.747\log q + 1.53 + \frac{1.638}{\eul(q)}
 + \frac{2\log(2\pi)}{\pi\beta}\Big(1 - \frac{1}{\beta}\Big)
 + \frac{1}{\beta}                                                                              \\
&+ \frac{2\cdot 2.89}{T}
 + \Big(\frac{1}{\pi}\log\Big(\frac{2}{\pi}\frac{q^2\sqrt{x}}{\eul(q)}\Big)\log\Big(\frac{2}{\pi}\frac{\sqrt{x}}{\eul(q)}\Big)
        + 3.9\log q  + 0.7 + \frac{22.7}{\eul(q)}
   \Big)\frac{h}{x}                                                                             \\
&+ \Big(\frac{1.58 \pi}{2} + \frac{16.08}{\beta}\Big)\frac{h}{\beta x}
 + \frac{3.7}{x^{1/6}}
 + \omega(q)\frac{\log(2x)}{\sqrt{x}}
 + \frac{h}{\pi x}\Big(1 + \frac{2.89}{T}\Big)\frac{\log(qT)}{T}
 + 1.7\frac{h}{x^{3/2}}.
\end{align*}
In several places we have assumed $T\geq 20$, thus we can use this assumption to note that it implies
\[
\frac{1}{\pi}\Big(1 + \frac{2.89}{T}\Big)\frac{\log(qT)}{T}
\leq 0.02\log q + 0.06.
\]
%
%
We further assume $x\geq (8\eul(q)\log q\log\log q)^2$ to bound
\begin{equation}\label{eq:A47}
   \log\Big(\frac{2}{\pi}\Big)
 + 2.53
 + \frac{1.638}{\eul(q)}
 + \frac{2\log(2\pi)}{\pi\beta}\Big(1 - \frac{1}{\beta}\Big)
 + \frac{1}{\beta}
 + \frac{2\cdot 2.89}{T}
 + \frac{3.7}{x^{1/6}}
 + \omega(q)\frac{\log(2x)}{\sqrt{x}}
 + 1.7\frac{h}{x^{3/2}}
\end{equation}
%
with ${\mathcal E}(q)$, which is $9.3$ when $q\leq 12$ and $4$ otherwise.
%
%
%
%
%
Hence it is sufficient to have
\begin{equation}\label{eq:A48}
\begin{aligned}
\rho
\geq&
   {\mathcal E}(q)
 - 0.747\log q
\\
&+ \Big(\frac{1}{\pi}\log\Big(\frac{2}{\pi}\frac{q^2\sqrt{x}}{\eul(q)}\Big)\log\Big(\frac{2}{\pi}\frac{\sqrt{x}}{\eul(q)}\Big)
        + 3.92\log q + 0.76 + \frac{22.7}{\eul(q)}
   \Big)\frac{h}{x}
 + \Big(\frac{1.58 \pi}{2} + \frac{16.08}{\beta}\Big)\frac{h}{\beta x}.
\end{aligned}
\end{equation}
%
%
Recalling the definitions of $h$ and $\beta$, \eqref{eq:A48} becomes:
\begin{equation}\label{eq:A49}
(1-F(q,x))\rho \geq G(q,x)
\end{equation}
with
\begin{align*}
F(q,x):=&\Big(\frac{1}{\pi}\log\Big(\frac{2}{\pi}\frac{q^2\sqrt{x}}{\eul(q)}\Big)\log\Big(\frac{2}{\pi}\frac{\sqrt{x}}{\eul(q)}\Big)
              + 3.92\log q + 0.76 + \frac{22.7}{\eul(q)}
         \Big)\frac{\eul(q)}{\sqrt{x}}                                                                  \\
        &+ \Big(\frac{1.58}{2} + \frac{16.08}{\log(q^2x)}\Big)\frac{\pi^2 \eul(q)}{\log(q^2x)\sqrt{x}}, \\
G(q,x):=& {\mathcal E}(q)
         + \Big(\frac{1}{\pi}\log\Big(\frac{2}{\pi}\frac{q^2\sqrt{x}}{\eul(q)}\Big)\log\Big(\frac{2}{\pi}\frac{\sqrt{x}}{\eul(q)}\Big)
                + 3.92\log q
           \Big)\log(q\sqrt{x})\frac{\eul(q)}{\sqrt{x}}
         - 0.747\log q                                                                                  \\
        &+ \Big(0.76 + \frac{22.7}{\eul(q)}\Big)\log(q^2x)\frac{\eul(q)}{2\sqrt{x}}
         + \Big(\frac{1.58}{2} + \frac{16.08}{\log(q^2x)}\Big)\frac{\pi^2\eul(q)}{2\sqrt{x}}.
\end{align*}
%
%
%
We notice that $F(q,x)$ and $G(q,x)$ decrease as a function of $x$ (hence there is no need to change $G$,
in this case), at least for $x\geq e^6 = 403.42\ldots$. Thus, if~\eqref{eq:A49} holds for fixed $\rho$
and $q$, for a given $x_0(q)$, then it holds for any $x\geq x_0(q)$ for the same $\rho$ and $q$.\\
Moreover we have to satisfy the assumptions
\begin{align}
\frac{1}{T} &= \frac{h}{\beta x}
            = \frac{\pi \eul(q)}{\sqrt{x}}\Big(\frac{1}{2}+\frac{\rho}{\log(q^2x)}\Big)
            \leq \frac{1}{20}                                                  \label{eq:A50}
\intertext{and}
\frac{h}{x} &= \frac{\eul(q)}{\sqrt{x}}\Big(\frac{1}{2}\log(q^2x)+\rho\Big)
           \leq \frac{5}{6},                                                   \label{eq:A51}
\end{align}
where again the functions appearing on the left hand side decrease in $x$.\\
We verify by direct computation that all these requirements are satisfied for $\rho=15$ by any $x\geq
x_0(q)$ with $x_0(q)$ given in Table~\ref{tab:A7}, when $q\leq 660$. For this purpose, we use a variant
of Procedure~\texttt{Check1}.

To deal with larger $q$'s, we choose $x_0(q):= (m\eul(q)\ell(q))^2$, where we set $\ell(q):=\log
q\log\log q$ to simplify the notation. To select a suitable value for $m$ we note that
$G_0(q):=G(q,x_0(q))$ stays bounded if and only if
\[
\frac{1}{\pi}\log\Big(\frac{q^2\sqrt{x_0(q)}}{\eul(q)}\Big)\log\Big(\frac{\sqrt{x_0(q)}}{\eul(q)}\Big)
\log(q\sqrt{x_0(q)})\frac{\eul(q)}{\sqrt{x_0(q)}}
- 0.747\log q
\]
is bounded, and that this happens if and only if $\frac{4}{\pi m} < 0.747$.
%
%
This shows that any $m$ larger than $2$, say, is allowed when $q\geq q_0(m)$ is large enough. With this
choice of $x_0(q)$, inequalities~\eqref{eq:A50} and~\eqref{eq:A51} are satisfied as soon as
\begin{equation}\label{eq:A52}
\frac{\pi}{m\ell(q)}\Big(\frac{1}{2}+\frac{\rho/2}{\log(m q\eul(q)\ell(q))}\Big)
\leq \frac{1}{20},
\end{equation}
and
\begin{equation}\label{eq:A53}
\frac{1}{m\ell(q)}\big(\log(mq\eul(q)\ell(q))+\rho\big)
\leq \frac{5}{6}.
\end{equation}
To deal with~\eqref{eq:A49}, \eqref{eq:A52} and \eqref{eq:A53} for arbitrary $q$ we substitute there the
arithmetical function $\eul(q)$ with its upper bound $q$ or its lower bound $\sqrt{q}$ in order to
produce in any case upper-bounds $\tilde{F}_0(q)$ and $\tilde{G}_0(q)$ for $F_0(q):=F(q,x_0(q))$ and
$G_0(q)$ respectively, and for the function to the left hand side of~\eqref{eq:A52}. In this
way~\eqref{eq:A49} changes into
\begin{equation}\label{eq:A54}
(1-\tilde{F}_0)\rho \geq \tilde{G}_0.
\end{equation}
As for Theorem~\ref{th:A1}, functions $\tilde{F}_0$ and those we get from~\eqref{eq:A52}
and~\eqref{eq:A53} are decreasing in $q$, while this remains false for $\tilde{G}_0$. However, contrary
to the situation for Theorem~\ref{th:A1} the parameters $\alpha$ ($=\frac{1}{2}$), $\delta$ ($=1$) and
$\rho$ ($=15$) are now fixed, thus we can verify directly that $\rho \geq -\tilde{G}_0'/\tilde{F}_0'$
for any $q\geq 3$ and any integer $8\leq m\leq 20$. This shows that for these parameters
$(1-\tilde{F}_0)\rho - \tilde{G}_0$ is increasing in the full range for $q$.

In this way we can conclude that when $8\leq m\leq 20$ all conditions we have to test become monotonous
in their dependence of $x$ and $q$, so that we can prove them for $x\geq x_0(q)$ and $q\geq q_0(m)$ by
proving them for $x=x_0(q)$ and $q=q_0(m)$. We have collected some results in Table~\ref{tab:A8}, for
several values of $m$. We see that the value $m=8$ produces a small enough $q_0(m)$, hence we have
selected it, as reported in
Theorem~\ref{th:A2}.
To complete the proof of Theorem~\ref{th:A2} we still need to test the claim for $3\leq q<660$ and $x$
in the interval $[(8\eul(q)\ell(q))^2,x_0(q)]$ with $x_0(q)$ given in Table~\ref{tab:A7}. For
this purpose we use an analogue of Procedure~{\tt Check}.

\medskip

%
For the second part of the theorem it is sufficient to prove that the right hand side of~\eqref{eq:A23}
is larger than $\log(x+h)$ when we increasing $h$ to $h + \eul(q)\sqrt{x}\log x$. This
modifies~\eqref{eq:A49} into
\[
(1-F(q,x))\rho \geq G(q,x) + F(q,x)\log x + \log(1+5/6) =:G_s(q,x).
\]
%
%
We proceed as before. In fact, both sides are decreasing as a function of $x$. Thus, we verify by direct
computation that all these requirements are satisfied for $\rho=15$ by any $x\geq x'_0(q)$ with
$x'_0(q)$ given in Table~\ref{tab:A7}, when $q\leq 1320$.

Again, we choose $x_0(q):= (m'\eul(q)\ell(q))^2$, producing
\begin{equation}\label{eq:A55}
(1-F(q,x'_0(q)))\rho
\geq G(q,x'_0(q)) + F(q,x'_0(q))\log(x'_0(q)) + \log(1+5/6) = G_s(q,x'_0(q)).
\end{equation}
In order to have $G_s(q,x_0(q))$ bounded it is necessary that $\frac{8}{\pi m'} < 0.747$, thus any $m'\geq
4$ suffices. With this choice of $x'_0(q)$, inequalities~\eqref{eq:A50} and~\eqref{eq:A51} are satisfied as
soon as
\begin{equation}\label{eq:A56}
\frac{\pi}{m'\ell(q)}
\Big(\frac{1}{2}
     +\frac{\rho/2+\log(m'\eul(q)\ell(q))}{\log(m' q\eul(q)\ell(q))}
\Big)
\leq \frac{1}{20},
\end{equation}
and
\begin{equation}\label{eq:A57}
\frac{1}{m'\ell(q)}\big(\log(m'q\eul(q)\ell(q))+2\log(m'\eul(q)\ell(q))+\rho\big)
\leq \frac{5}{6}.
\end{equation}
To deal with~\eqref{eq:A55}, \eqref{eq:A56} and \eqref{eq:A57} for arbitrary $q$ we substitute there
the arithmetical function $\eul(q)$ with its upper bound $q$ or its lower bound $\sqrt{q}$ in order to
produce in any case upper-bounds $\tilde{F}(q,x_0(q))$ and $\tilde G_s(q,x_0(q))$ for $F(q,x_0(q))$ and
$G_s(q,x_0(q))$ respectively, and for the function on the left hand side of~\eqref{eq:A56}.
In this way~\eqref{eq:A55} changes into
\[
(1-\tilde{F})\rho \geq \tilde G_s.
\]
Functions $\tilde{F}$, and those we get from~\eqref{eq:A56} and~\eqref{eq:A57} are evidently decreasing
in $q$, but this is still false for $\tilde G_s$. However, $(1-\tilde{F})\rho - \tilde G_s$ is decreasing
if and only if $\rho \geq -\tilde G_s'/\tilde{F}'$ and for $\rho=15$ this holds for any $q\geq 3$ if
$m'\geq 10$.
In this way we can conclude that when $m'\geq 10$ all conditions we have to test become monotonous in
their dependence of $x$ and $q$, so that we can prove them for $x\geq x'_0(q)$ and $q\geq q'_0(m')$ by
proving them for $x=x'_0(q)$ and $q=q'_0(m')$. We have collected some results in Table~\ref{tab:A8}, for
several values of $m'$. Unfortunately, the computations show that any value of $m'$ smaller than $15$
would produce an extremely large $q_0(m')$. As a consequence we have selected $m'=15$, as reported in
Theorem~\ref{th:A2}.

\bigskip

%
Lastly, it is easy to prove that $F(q,e^q)$ is smaller than $1$ for $q\geq 10$ and that $G(q,e^q)\leq 0$
for $q\geq 220$, and $G_s(q,e^q)\leq 0$ for $q\geq 500$ and this proves Theorem~\ref{th:A3} with $q\geq
220$ for the first claim and $q\geq 500$ for the second.
%
%
%
%
%
%
%
The first (second) claim is extended to $q\geq 35$ ($q\geq 67$, respectively) keeping the true value
of~\eqref{eq:A47} in place of $\mathcal{E}(q)$ in the definition of $G(q,x)$.
%
%
%
%

\section{Proof of Corollary~\ref{cor:A1}}
\label{sec:A7}
We can assume $q\geq 3$, because the claim for $q=1$ and $q=2$ follows from the analogous (and stronger)
claim proved in~\cite[Cor.~4.1]{DudekGrenieMolteni1}.\\
By Theorem~\ref{th:A1} (case $\alpha=1/2$, $\delta=1$) we know that there is a prime congruent to $a$
modulo $q$ as soon as
\[
(2n+\eul(q)A)\eul(q)A
\geq \eul(q)\sqrt{M}(24+\log(q^2M))
\]
where $A:=12+2\log(qn)$ and $M := \frac{1}{2}[n^2 + (n+\eul(q)A)^2]$.
Dividing by $nA$ and setting $B:=\frac{\eul(q)}{n}A$, the inequality becomes
\[
2+B
\geq \sqrt{1+B+\frac{B^2}{2}}\Big(1 + \frac{12+\log(1+B+\frac{B^2}{2})}{A}\Big),
\]
i.e.,
\[
\sqrt{\frac{4+4B+2B^2}{4+4B+B^2}}\Big(1 + \frac{12+\log(1+B+\frac{B^2}{2})}{A}\Big) \leq 2.
\]
Set $H:=\sqrt{\frac{4+4B+2B^2}{4+4B+B^2}}$, and notice that it is an increasing function of $B$, and is
bounded by $\sqrt{2}$. Hence the inequality may be written as
\[
1+B+\frac{B^2}{2}\leq \exp\Big(A\Big(\frac{2}{H}-1\Big) - 12\Big).
\]
In terms of $B$ this is solved by
\[
B \leq \Big[2\exp\Big(A\Big(\frac{2}{H}-1\Big) - 12\Big)-1\Big]^{1/2} - 1,
\]
but needs $2\exp(A\big(\frac{2}{H}-1\big) - 12)\geq 1$.
Recalling the definition of $B$, it means that
\[
\eul(q) \leq \frac{n}{A}\Big[\Big[2\exp\Big(A\Big(\frac{2}{H}-1\Big) - 12\Big)-1\Big]^{1/2} - 1\Big].
\]
Recalling the definition of $A$, we see that for every fixed value of $q$, the quotient $n/A$ increases
with $n$. Hence $B=\eul(q)\frac{A}{n}$ decreases with $n$, and $1/H$ (which decreases with $B$) increases
with $n$. This shows that the function appearing on the right hand side increases as a function of $n$,
for every fixed $q$, if $A(2/H-1)\geq 12$. As a consequence the inequalities hold true for $n\geq n_0$ as
soon they hold for $n=n_0$. It is easy to prove that for $n\geq 8\eul(q)\log q$ they hold for all $q\geq
3$.

\newpage
\section{Auxiliary tables}
\label{sec:A8}
\begin{table}[H]
\caption{Parameters for $q\to\infty$.}\label{tab:A3}
\centering
{
\begin{tabular}{|rrrr|rr||rrrr|rr|}
  \toprule
 \mltc{1}{|c}{$\alpha$}&\mltc{1}{c}{$\delta$}&\mltc{1}{c}{$m$}&\mltc{1}{c|}{$\ell$}&\mltc{1}{c}{$m'$}&\mltc{1}{c||}{$\ell'$}
&\mltc{1}{|c}{$\alpha$}&\mltc{1}{c}{$\delta$}&\mltc{1}{c}{$m$}&\mltc{1}{c|}{$\ell$}&\mltc{1}{c}{$m'$}&\mltc{1}{c||}{$\ell'$}\\
  \midrule
    1/2 & 1   &    21 &  7 &   44 &  6 & 1.253/2 & 0.1 &   142 &   17 &   373 &   17 \\
    1/2 & 1/2 &    56 &  7 &  139 &  7 &   1     & 0   &    21 &    7 &    44 &    6 \\
    1/2 & 1/3 &   179 & 24 &  475 & 21 &   0.9   & 0   &    27 &    7 &    60 &    5 \\
1.253/2 & 1   &    17 &  8 &   34 &  6 &   0.8   & 0   &    40 &    7 &    97 &    5 \\
1.253/2 & 1/2 &    29 &  6 &   66 &  5 &   0.7   & 0   &    95 &   11 &   245 &   10 \\
1.253/2 & 0.2 &    69 &  9 &  175 &  8 &   0.627 & 0   & 21236 & 1652 & 57287 & 1310 \\
  \bottomrule
\end{tabular}
}
\end{table}

\begin{table}[H]
\caption{Parameters}\label{tab:A4}
\centering
\begin{tabular}{|rrrrrr|rrr|}
  \toprule
 \mltc{1}{|c}{$\alpha$}&\mltc{1}{c}{$\delta$}&\mltc{1}{c }{$\rho$}
&\mltc{1}{ c}{$m$}     &\mltc{1}{c}{$\ell$}  &\mltc{1}{c|}{$q_0$}
&\mltc{1}{ c}{$m'$}    &\mltc{1}{c}{$\ell'$} &\mltc{1}{c|}{$q'_0$}\\
  \midrule
 1/2     & 1   & 12 &    23    &    6.4&    1947657 &             46 &   5.3&   1984065\\
 1/2     & 1/2 &  9 &    86    &   14  &     443235 &            188 &  11  &   2974713\\
 1/2     & 1/3 &  9 &  1500    &  120  &    2293436 &           3500 & 190  &   2711303\\

 1.253/2 & 1   & 14 &    18    &    7  &    7991888 &             34 &   5.7&   6306843\\
 1.253/2 & 1/2 &  9 &    34    &    7  &    3055181 &             74 &   6  &    920941\\
 1.253/2 & 0.2 &  7 &   110    &   18  &    3287890 &            260 &  15  &   3790727\\
 1.253/2 & 0.1 &  7 &   500    &   64  &    2878356 &           1500 &  66  &    999372\\

 1       & 0   &  8 &    23    &    6.4&    1972765 &             46 &   5.3&  2001416\\
 0.9     & 0   &  7 &    31    &    6  &    2617343 &             66 &   5  &  1294983\\
 0.8     & 0   &  6 &    52    &    9  &    1987447 &            120 &   8  &   630195\\
 0.7     & 0   &  5 &   200    &   16  &    1713915 &            500 &  26  &   958214\\
 0.627   & 0   & 10 &$10^{10}$ & 3480  & $10^{438}$ &       $10^{10}$&4100  &$10^{438}$\\
  \bottomrule
\end{tabular}
\end{table}

\newpage
{
\begin{table}[H]
\caption{Exceptions: for these $q$'s the claim has to be tested in $[x_0(q),x(q)]$}\label{tab:A5}
\centering
\begin{tabular}{|rrr|rrr|rrr|}
  \toprule
 \mltc{9}{|c|}{$\alpha=1/2$, $\delta=1$, $\rho=12$, $m=23$, $\ell=6.4$}\\
 \mltc{1}{|c}{$q$}&\mltc{1}{c}{$x_0(q)$}&\mltc{1}{c|}{$x(q)$}&
 \mltc{1}{ c}{$q$}&\mltc{1}{c}{$x_0(q)$}&\mltc{1}{c|}{$x(q)$}&
 \mltc{1}{ c}{$q$}&\mltc{1}{c}{$x_0(q)$}&\mltc{1}{c|}{$x(q)$} \\
   3 &   2553 & 23000 &   6 &   6793 & 23000 &      9 & 91940 &  94714\\
   4 &   4066 & 23000 &   7 &  72111 & 81124 &     10 & 44875 &  55094\\
   5 &  21924 & 37494 &   8 &  36598 & 51147 &     12 & 52263 &  60595\\
  \midrule
 \mltc{9}{|c|}{$\alpha=1/2$, $\delta=1/2$, $\rho=9$, $m=86$, $\ell=14$}\\
 \mltc{1}{|c}{$q$}&\mltc{1}{c}{$x_0(q)$}&\mltc{1}{c|}{$x(q)$}&
 \mltc{1}{ c}{$q$}&\mltc{1}{c}{$x_0(q)$}&\mltc{1}{c|}{$x(q)$}&
 \mltc{1}{ c}{$q$}&\mltc{1}{c}{$x_0(q)$}&\mltc{1}{c|}{$x(q)$} \\
   3 & 35706 &  77348 &   4 & 56854 & 95500 &      6 & 94976 &  104272\\
  \midrule
 \mltc{9}{|c|}{$\alpha=1/2$, $\delta=1/3$, $\rho=9$, $m=1500$, $\ell=120$: no exceptions}\\
  \midrule
 \mltc{9}{|c|}{$\alpha=1.253/2$, $\delta=1$, $\rho=14$, $m=18$, $\ell=7$}\\
 \mltc{1}{|c}{$q$}&\mltc{1}{c}{$x_0(q)$}&\mltc{1}{c|}{$x(q)$}&
 \mltc{1}{ c}{$q$}&\mltc{1}{c}{$x_0(q)$}&\mltc{1}{c|}{$x(q)$}&
 \mltc{1}{ c}{$q$}&\mltc{1}{c}{$x_0(q)$}&\mltc{1}{c|}{$x(q)$} \\
   3 &    1564 & 174459  &   6 &   4160 &    23000 &      9 & 56311 & 59241 \\
   4 &    2490 & 190024  &   7 &  44166 &    50277 &     10 & 27485 & 35009 \\
   5 &   13428 & 565474  &   8 &  22416 &    31807 &     12 & 32009 & 38677 \\
  \midrule
 \mltc{9}{|c|}{$\alpha=1.253/2$, $\delta=1/2$, $\rho=9$, $m=34$, $\ell=7$}\\
 \mltc{1}{|c}{$q$}&\mltc{1}{c}{$x_0(q)$}&\mltc{1}{c|}{$x(q)$}&
 \mltc{1}{ c}{$q$}&\mltc{1}{c}{$x_0(q)$}&\mltc{1}{c|}{$x(q)$}&
 \mltc{1}{ c}{$q$}&\mltc{1}{c}{$x_0(q)$}&\mltc{1}{c|}{$x(q)$} \\
   3 &   5580 &   24333  &  5 &  47910 &  62458 &      8 & 79978 &    87897\\
   4 &   8886 &   29766  &  6 &  14844 &  34684&&&\\
  \midrule
 \mltc{9}{|c|}{$\alpha=1.253/2$, $\delta=0.2$, $\rho=7$, $m=110$, $\ell=18$}\\
 \mltc{1}{|c}{$q$}&\mltc{1}{c}{$x_0(q)$}&\mltc{1}{c|}{$x(q)$}&
 \mltc{1}{ c}{$q$}&\mltc{1}{c}{$x_0(q)$}&\mltc{1}{c|}{$x(q)$}&
 \mltc{1}{ c}{$q$}&\mltc{1}{c}{$x_0(q)$}&\mltc{1}{c|}{$x(q)$} \\
   3 &  58416 & 136773  &      4 &  93015 &  176298 &      6 &  155383 &  196485\\
  \midrule
 \mltc{9}{|c|}{$\alpha=1.253/2$, $\delta=0.1$, $\rho=7$, $m=500$, $\ell=64$: no exceptions}\\
  \midrule
 \mltc{9}{|c|}{$\alpha=1$, $\delta=0$, $\rho=8$, $m=23$, $\ell=6.4$}\\
 \mltc{1}{|c}{$q$}&\mltc{1}{c}{$x_0(q)$}&\mltc{1}{c|}{$x(q)$}&
 \mltc{1}{ c}{$q$}&\mltc{1}{c}{$x_0(q)$}&\mltc{1}{c|}{$x(q)$}&
 \mltc{1}{ c}{$q$}&\mltc{1}{c}{$x_0(q)$}&\mltc{1}{c|}{$x(q)$} \\
   3 &     2553 &      23000  &      6 &     6793 &      23000 &     10 &   44875 &      52243\\
   4 &     4066 &      23000  &      8 &    36598 &      47072 &     12 &   52263 &      58690\\
   5 &    21924 &      32725  &&&&&&\\
  \midrule
 \mltc{9}{|c|}{$\alpha=0.9$, $\delta=0$, $\rho=7$, $m=31$, $\ell=6$}\\
 \mltc{1}{|c}{$q$}&\mltc{1}{c}{$x_0(q)$}&\mltc{1}{c|}{$x(q)$}&
 \mltc{1}{ c}{$q$}&\mltc{1}{c}{$x_0(q)$}&\mltc{1}{c|}{$x(q)$}&
 \mltc{1}{ c}{$q$}&\mltc{1}{c}{$x_0(q)$}&\mltc{1}{c|}{$x(q)$} \\
   3 &   4639 &      23000  &      5 &  39828 &      47524 &     8 &  66487 &      69419\\
   4 &   7387 &      23032  &      6 &  12340 &      28176 &     &&\\
  \midrule
 \mltc{9}{|c|}{$\alpha=0.8$, $\delta=0$, $\rho=6$, $m=52$, $\ell=9$}\\
 \mltc{1}{|c}{$q$}&\mltc{1}{c}{$x_0(q)$}&\mltc{1}{c|}{$x(q)$}&
 \mltc{1}{ c}{$q$}&\mltc{1}{c}{$x_0(q)$}&\mltc{1}{c|}{$x(q)$}&
 \mltc{1}{ c}{$q$}&\mltc{1}{c}{$x_0(q)$}&\mltc{1}{c|}{$x(q)$} \\
   3 &   13054&      45973  &      5 &  112066&     116443 &     6 &   34723&      70349\\
   4 &   20786&      58793  &&&&&&\\
  \midrule
 \mltc{9}{|c|}{$\alpha=0.7$, $\delta=0$, $\rho=5$, $m=200$, $\ell=16$}\\
 \mltc{1}{|c}{$q$}&\mltc{1}{c}{$x_0(q)$}&\mltc{1}{c|}{$x(q)$}&
 \mltc{1}{ c}{$q$}&\mltc{1}{c}{$x_0(q)$}&\mltc{1}{c|}{$x(q)$}&
 \mltc{1}{ c}{$q$}&\mltc{1}{c}{$x_0(q)$}&\mltc{1}{c|}{$x(q)$} \\
   3 &   193111&     283439 &     4 & 307489&     391345 &&&\\
  \bottomrule
\end{tabular}
\end{table}
}

{\scriptsize
\begin{table}[H]
\caption{Exceptions: for these $q$'s the claim has to be tested in $[x'_0(q),x'(q)]$}\label{tab:A6}
\centering
\begin{tabular}{|rrr|rrr|rrr|}
  \toprule
 \mltc{9}{|c|}{$\alpha=1/2$, $\delta=1$, $\rho=12$, $m'=46$, $\ell'=5.3$}\\
 \mltc{1}{|c}{$q$}&\mltc{1}{c}{$x'_0(q)$}&\mltc{1}{c|}{$x'(q)$}&
 \mltc{1}{ c}{$q$}&\mltc{1}{c}{$x'_0(q)$}&\mltc{1}{c|}{$x'(q)$}&
 \mltc{1}{ c}{$q$}&\mltc{1}{c}{$x'_0(q)$}&\mltc{1}{c|}{$x'(q)$} \\
    3 &  10215 & 28413    &     4 & 16266 & 33887    &    6 & 27172 & 39233\\
  \midrule
 \mltc{9}{|c|}{$\alpha=1/2$, $\delta=1/2$, $\rho=9$, $m'=188$, $\ell'=11$: no exceptions}\\
  \midrule
 \mltc{9}{|c|}{$\alpha=1/2$, $\delta=1/3$, $\rho=9$, $m'=3500$, $\ell'=190$: no exceptions}\\
  \midrule
 \mltc{9}{|c|}{$\alpha=1.253/2$, $\delta=1$, $\rho=14$, $m'=34$, $\ell'=5.7$}\\
 \mltc{1}{|c}{$q$}&\mltc{1}{c}{$x'_0(q)$}&\mltc{1}{c|}{$x'(q)$}&
 \mltc{1}{ c}{$q$}&\mltc{1}{c}{$x'_0(q)$}&\mltc{1}{c|}{$x'(q)$}&
 \mltc{1}{ c}{$q$}&\mltc{1}{c}{$x'_0(q)$}&\mltc{1}{c|}{$x'(q)$} \\
   3 &  5580 &     23000    &     4 &  8886 &      23000  &     6 & 14844 &  23000\\
  \midrule
 \mltc{9}{|c|}{$\alpha=1.253/2$, $\delta=1/2$, $\rho=9$, $m'=74$, $\ell'=6$}\\
 \mltc{1}{|c}{$q$}&\mltc{1}{c}{$x'_0(q)$}&\mltc{1}{c|}{$x'(q)$}&
 \mltc{1}{ c}{$q$}&\mltc{1}{c}{$x'_0(q)$}&\mltc{1}{c|}{$x'(q)$}&
 \mltc{1}{ c}{$q$}&\mltc{1}{c}{$x'_0(q)$}&\mltc{1}{c|}{$x'(q)$} \\
   3 &  26437 &     53359    &     4 &  42095 &      65485  &     6 & 70320 &  76541\\
  \midrule
 \mltc{9}{|c|}{$\alpha=1.253/2$, $\delta=0.2$, $\rho=7$, $m'=260$, $\ell'=15$: no exceptions}\\
  \midrule
 \mltc{9}{|c|}{$\alpha=1.253/2$, $\delta=0.1$, $\rho=7$, $m'=1500$, $\ell'=66$: no exceptions}\\
  \midrule
 \mltc{9}{|c|}{$\alpha=1$, $\delta=0$, $\rho=8$, $m'=46$, $\ell'=5.3$}\\
 \mltc{1}{|c}{$q$}&\mltc{1}{c}{$x'_0(q)$}&\mltc{1}{c|}{$x'(q)$}&
 \mltc{1}{ c}{$q$}&\mltc{1}{c}{$x'_0(q)$}&\mltc{1}{c|}{$x'(q)$}&
 \mltc{1}{ c}{$q$}&\mltc{1}{c}{$x'_0(q)$}&\mltc{1}{c|}{$x'(q)$} \\
   3 &  10215&      26091   &      4 &  16266&      32379  &      6 & 27172&      39992\\
  \midrule
 \mltc{9}{|c|}{$\alpha=0.9$, $\delta=0$, $\rho=7$, $m'=66$, $\ell'=5$}\\
 \mltc{1}{|c}{$q$}&\mltc{1}{c}{$x'_0(q)$}&\mltc{1}{c|}{$x'(q)$}&
 \mltc{1}{ c}{$q$}&\mltc{1}{c}{$x'_0(q)$}&\mltc{1}{c|}{$x'(q)$}&
 \mltc{1}{ c}{$q$}&\mltc{1}{c}{$x'_0(q)$}&\mltc{1}{c|}{$x'(q)$} \\
   3 &  21029&      40486   &      4 &  33485&      50922 &     6 & 55938&      62679\\
  \midrule
 \mltc{9}{|c|}{$\alpha=0.8$, $\delta=0$, $\rho=6$, $m'=120$, $\ell'=8$}\\
 \mltc{1}{|c}{$q$}&\mltc{1}{c}{$x'_0(q)$}&\mltc{1}{c|}{$x'(q)$}&
 \mltc{1}{ c}{$q$}&\mltc{1}{c}{$x'_0(q)$}&\mltc{1}{c|}{$x'(q)$}&
 \mltc{1}{ c}{$q$}&\mltc{1}{c}{$x'_0(q)$}&\mltc{1}{c|}{$x'(q)$} \\
   3 &  69520&     108608    &      4 &   110696&     139012 &&&\\
  \midrule
 \mltc{9}{|c|}{$\alpha=0.7$, $\delta=0$, $\rho=5$, $m'=500$, $\ell'=26$: no exceptions}\\
  \bottomrule
\end{tabular}
\end{table}
}

{
\begin{table}[H]
\caption{Constants for the proof of Theorem~\ref{th:A2}: small $q$'s.}\label{tab:A7}
\centering
\small
\begin{tabular}{|rr|rr||rr|rr|}
  \toprule
\mltc{1}{|c}{$q$}&\mltc{1}{c|}{$x_0$} & \mltc{1}{|c}{$q$}&\mltc{1}{c||}{$x_0$} & \mltc{1}{c}{$q$}&\mltc{1}{c|}{$x'_0$} & \mltc{1}{|c}{$q$}&\mltc{1}{c|}{$x'_0$}\\
  \midrule
 3 &  43741 &  9                   &  273368                & 3 &  98197 &  9                   &  826355                \\
 4 &  41398 & 10                   &  126848                & 4 & 108188 & 10                   &  419894                \\
 5 & 141162 & 11                   &  690311                & 5 & 317506 & 11                   & 2381080                \\
 6 &  38467 & 12                   &  126684                & 6 & 122626 & 12                   &  447783                \\
 7 & 283378 & $13 \leq q\leq  100$ & $(34\eul(q)\ell(q))^2$ & 7 & 739830 & $13 \leq q\leq  100$ & $(52\eul(q)\ell(q))^2$ \\
 8 & 131137 & $100\leq q\leq  660$ & $(10\eul(q)\ell(q))^2$ & 8 & 386260 & $100\leq q\leq 1320$ & $(20\eul(q)\ell(q))^2$ \\
  \bottomrule
\end{tabular}
\end{table}
}

\begin{table}[H]
\caption{Constants for the proof of Theorem~\ref{th:A2}: large $q$'s.}\label{tab:A8}
\centering
{
\begin{tabular}{|rr||rr|}
  \toprule
\mltc{1}{|c}{$m$}&\mltc{1}{c||}{$q_0$} & \mltc{1}{c}{$m'$}&\mltc{1}{c|}{$q'_0$}\\
  \midrule
  8 &         660 & 14 &    343072 \\
  9 &         168 & 15 &      1320 \\
 10 &         111 & 16 &       330 \\
  \bottomrule
\end{tabular}
}
\end{table}

\enlargethispage{23\baselineskip}
\SetKwFunction{Hone}{h1}
\begin{function}
  \KwIn{Three reals {$\alpha$, $\delta$, $\rho$}}
  \KwIn{Two integers {$q$, $x$}}
  \caption{h1($\alpha$, $\delta$, $\rho$, $q$, $x$)}
  \Return{$(\alpha\log x+\delta\log q+\rho)\eul(q)\sqrt{x}$}\;
\end{function}
\begin{procedure}
  \SetKwFor{Forprime}{forprime}{do}{endfp}
  \SetKw{Continue}{continue}
  \SetKwFunction{Floor}{floor}
  \SetKwFunction{Print}{print}
  \SetAlgoLined
  \KwIn{Three reals {$\alpha$, $\delta$, $\rho$}}
  \KwIn{Three integers {$q$, $x_0$, $x$}}
  \caption{Check1($\alpha$, $\delta$, $\rho$, $q$, $x_0$, $x$)}

  \For{$a\leftarrow 1$ \KwTo $q-1$}{
      \lIf{$(a,q)\neq 1$}{\Continue}
      $M_1[a]\leftarrow x_0+\text{\Hone}(\alpha,\delta,\rho,q,x_0)$\;
  }
  \Forprime{$p\leftarrow x_0$ \KwTo $x + \text{\Hone}(\alpha,\delta,\rho,q,x)$}{
      $a\leftarrow p\mod q$\;
      \lIf{$(a,q)\neq 1$}{\Continue}
      \If{$M_1[a]\leq p$}{
          \Print("Problem with class ", $a$, " mod ", $q$, " for $x=$", $M_1[a]$)\;
      }
      $M_1[a]\leftarrow p+\text{\Hone}(\alpha,\delta,\rho,q,p)$\;
  }
  \For{$a\leftarrow 1$ \KwTo $q-1$}{
      \lIf{$(a,q)\neq 1$}{\Continue}
      \If{$M_1[a]<x$}{
          \Print("Problem with class ", $a$, " mod ", $q$, " for $x=$", $M_1[a]$)\;
      }
  }
\end{procedure}
\SetKwFunction{Hsqrt}{hsqrt}
\begin{function}
  \KwIn{Three reals {$\alpha$, $\delta$, $\rho$}}
  \KwIn{Two integers {$q$, $x$}}
  \caption{hsqrt($\alpha$, $\delta$, $\rho$, $q$, $x$)}
  \Return{$((\alpha+1)\log x+\delta\log q+\rho)\eul(q)\sqrt{x}$}\;
\end{function}

\begin{procedure}
  \SetKwFor{Forprime}{forprime}{do}{endfp}
  \SetKw{Continue}{continue}
  \SetKwFunction{Floor}{floor}
  \SetKwFunction{Print}{print}
  \SetAlgoLined
  \KwIn{Three reals {$\alpha$, $\delta$, $\rho$}}
  \KwIn{Three integers {$q$, $x_0$, $x$}}
  \caption{CheckSqrt($\alpha$, $\delta$, $\rho$, $q$, $x'_0$, $x'$)}

  \For{$a\leftarrow 1$ \KwTo $q-1$}{
      \lIf{$(a,q)\neq 1$}{\Continue}
      $M_s[a]\leftarrow x'_0+\text{\Hsqrt}(\alpha,\delta,\rho,q,x'_0)$\;
      $N[a]\leftarrow \text{\Floor}(M_s[a])+1$\;
  }
  \Forprime{$p\leftarrow x'_0$ \KwTo $x' + \text{\Hsqrt}(\alpha,\delta,\rho,q,x')$}{
      $a\leftarrow p\mod q$\;
      \lIf{$(a,q)\neq 1$}{\Continue}
      $N[a]\leftarrow N[a]-1$\;
      \lIf{$N[a]\neq 0$}{\Continue}
      \If{$M_s[a]\leq p$}{
        \Print("Problem for sqrt claim with class ", $a$, " mod ", $q$, " for $x=$", $M_s[a]$)\;
      }
      $M_s[a]\leftarrow p+\text{\Hsqrt}(\alpha,\delta,\rho,q,p)$\;
      $N[a]\leftarrow \text{\Floor}(M_s[a])+1$\;
  }
  \For{$a\leftarrow 1$ \KwTo $q-1$}{
      \lIf{$(a,q)\neq 1$}{\Continue}
      \If{$M_s[a]<x'$}{
          \Print("Problem for sqrt claim with class ", $a$, " mod ", $q$, " for $x=$", $M_s[a]$)\;
      }
  }
\end{procedure}
\clearpage

\newpage
\begin{acknowledgements}
All computations have been done using PARI/GP~\cite{PARI2}. We have made available at the address:\\
\url{http://users.mat.unimi.it/~molteni/research/primes/progressions.gp}\\
the code we have used to compute the constants in this paper.
\end{acknowledgements}


\renewcommand{\thefootnote}{}\footnotetext{\today\ at \currenttime\par\jobname}
\end{document}